\documentclass[a4paper,12pt]{article}
\usepackage{amssymb,amsmath,amsthm,times,mathrsfs,fullpage}
\usepackage[pdftex]{hyperref}
\hypersetup{plainpages=True, pdfstartview=FitH, bookmarksopen=true,
pdfpagemode=none, colorlinks=true,linkcolor=blue,citecolor=blue}
\usepackage{pgf,tikz,yfonts,shadow}
\usetikzlibrary{scopes,shapes,snakes,arrows}
\usetikzlibrary{positioning,shadows,trees}

\newcommand{\SK}{\ensuremath{K}}
\newcommand{\SN}{\ensuremath{N}}
\def\le{\leqslant}
\def\ge{\geqslant}
\def\dd#1{{\,\mathrm{d}}#1}
\def\ve{\varepsilon}
\def\pf{\noindent \emph{Proof.}\ }
\def\qed{{\quad\rule{1mm}{3mm}\,}}

\newcommand{\E}{\ensuremath{\mathbb{E}}}
\newcommand{\V}{\ensuremath{\mathbb{V}}}
\newcommand{\Cov}{\ensuremath{\mathrm{Cov}}}
\newcommand{\Prob}{\ensuremath{\mathbb{P}}}
\newcommand{\N}{{\mathbb{N}}}
\newcommand{\Rset}{{\mathbb{R}}}

\newcommand{\R}{{\mathbb{R}}}
\newcommand{\C}{{\mathbb{C}}}

\begin{document}
\allowdisplaybreaks[1]

\newtheorem{thm}{Theorem}[section]
\newtheorem{cor}[thm]{Corollary}
\newtheorem{lem}[thm]{Lemma}
\newtheorem{conj}[thm]{Conjecture}
\newtheorem{pro}[thm]{Proposition}
\newtheorem{df}[thm]{Definition}
\definecolor{yg}{RGB}{235,255,204}

%\pagecolor{green!15}

\title{\textbf{Dependence and phase changes in random $m$-ary
search trees}}
\author{
Hua-Huai Chern\\
    Department of Computer Science\\
    National Taiwan Ocean University\\
    Keelung 202\\
    Taiwan \and
Michael Fuchs\thanks{Partially supported by the
	Ministry of Science and Technology, Taiwan
	under the grant MOST-103-2115-M-009-007-MY2.}\\
    Department of Applied Mathematics\\
    National Chiao Tung University\\
    Hsinchu 300\\
    Taiwan \and
Hsien-Kuei Hwang\thanks{This author's research stay at J.\ W.\
    Goethe-Universit\"at was partially supported
	by the Simons Foundation and by the Mathematisches
	Forschungsinstitut Oberwolfach.}\\
    Institute of Statistical Science\\
    Academia Sinica\\
    Taipei 115\\
    Taiwan  \and
Ralph Neininger\thanks{Supported by DFG grant NE 828/2-1.}\\
    Institute for Mathematics\\
    Goethe University\\
    60054 Frankfurt a.M.\\
    Germany
}
\date{\today}
\maketitle

\begin{abstract}

We study the joint asymptotic behavior of the space requirement and
the total path length (either summing over all root-key distances or
over all root-node distances) in random $m$-ary search trees. The
covariance turns out to exhibit a change of asymptotic behavior: it
is essentially linear when $3\le m\le 13$ but becomes of higher order
when $m\ge14$. Surprisingly, the corresponding asymptotic correlation
coefficient tends to zero when $3\le m\le 26$ but is periodically
oscillating for larger $m$, and we also prove asymptotic independence
when $3\le m\le 26$. Such a less anticipated phenomenon is not
exceptional and we extend the results in two directions: one for more
general shape parameters, and the other for other classes of random
log-trees such as fringe-balanced binary search trees and quadtrees.
The methods of proof combine asymptotic transfer for the underlying
recurrence relations with the contraction method.

\end{abstract}



\section{Introduction}

The $m$-ary search trees are a class of data structures introduced by
Muntz and Uzgalis \cite{muntz71} in 1971 in computer algorithms to
support efficient searching and sorting of data; see the next section
for more details. When constructed from a random permutation of $n$
elements, the space requirement (total number of nodes to store the
input) $S_n$ of such \emph{random $m$-ary search trees} ($m\ge3$) is
known to exhibit a \emph{phase change phenomenon}: its distribution
is asymptotically Gaussian for large $n$ when the branching factor
$m$ satisfies $3\le m\le 26$ but does not approach a limit law when
$m\ge27$; see \cite{chern01,hwang03, mahmoud92,mahmoud89} and the
references therein. On the other hand, it is also known that the
total key path length $K_n$ (the sum over all distances from the root
to any \emph{key}) does not change its limiting behavior when $m$
varies, and tends asymptotically, after properly centered and
normalized, to a limit law for each $m\ge3$. Another closely related
shape measure, the total node path length $N_n$ (summing over all
distances from the root to any \emph{node}) also follows
asymptotically a very similar behavior.

Our motivating question was ``how does $K_n$ or $N_n$ depend on
$S_n$?'' Surprisingly, despite the strong dependence of the
definition of $N_n$ on $S_n$ (see \eqref{S-N}), we show that the
correlation coefficient $\rho(S_n,N_n)$ satisfies
\begin{align}\label{rho-SN}
    \rho(S_n,N_n)
    \sim \begin{cases}
        0, &\text{if}\ 3\le m\le 26;\\ \displaystyle
        F_\rho(\beta\log n),
            &\text{if}\ m\ge 27,
    \end{cases}
\end{align}
where $F_\rho(t)$ is a $2\pi$-periodic function and $\beta=\beta_m$
is a structural constant depending on $m$. The same type of results
also holds for $\rho(S_n,K_n)$. In words, \emph{$N_n$ and $S_n$ are
asymptotically uncorrelated for $3\le m\le 26$ and their correlation
fluctuates (between $-1$ and $1$) for $m\ge27$}; see
Figure~\ref{fig-corr-coeff} for an illustration.

\begin{figure}[!ht]
\begin{center}
\includegraphics[height=3cm]{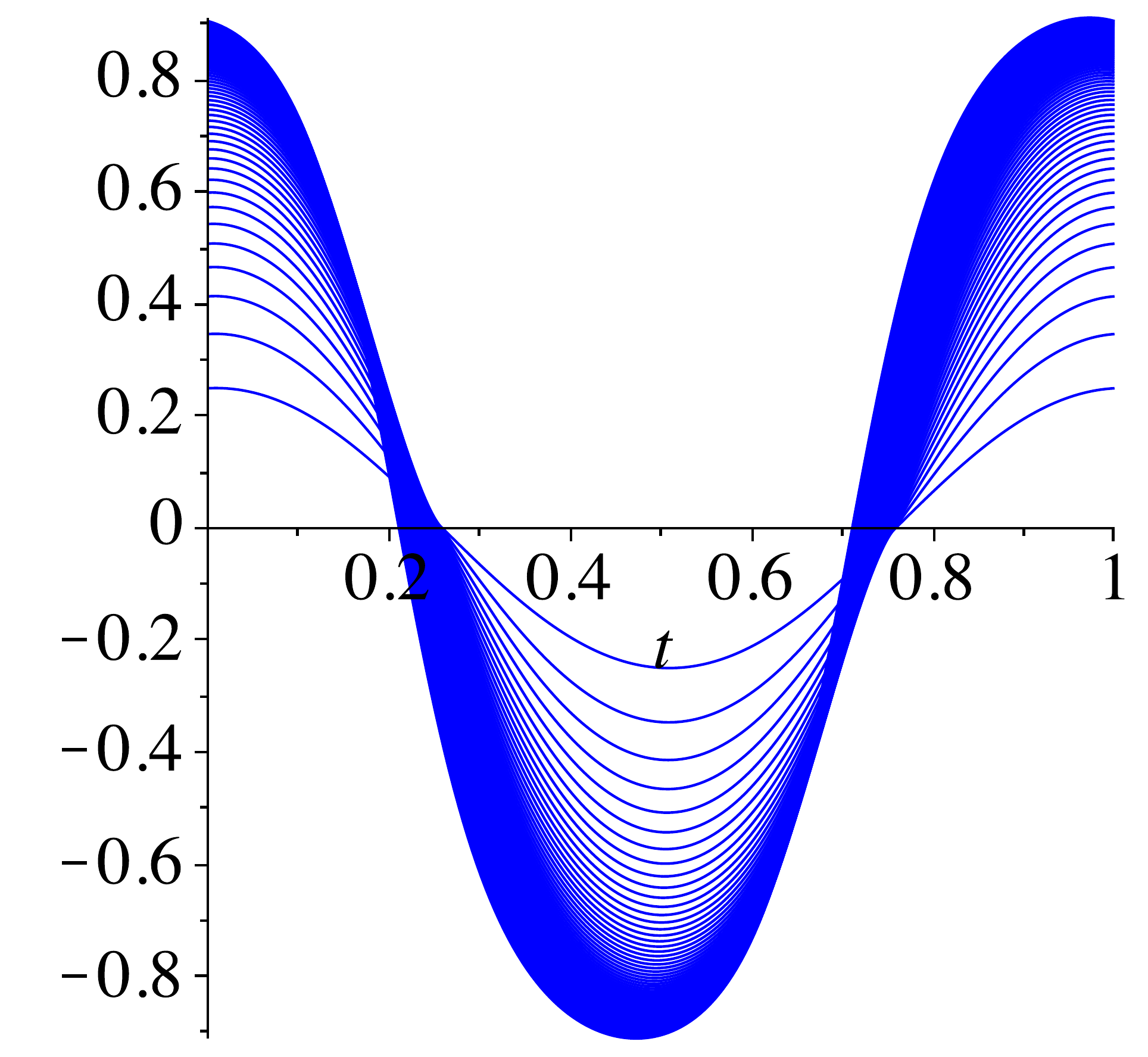}\;\;
\includegraphics[height=3cm]{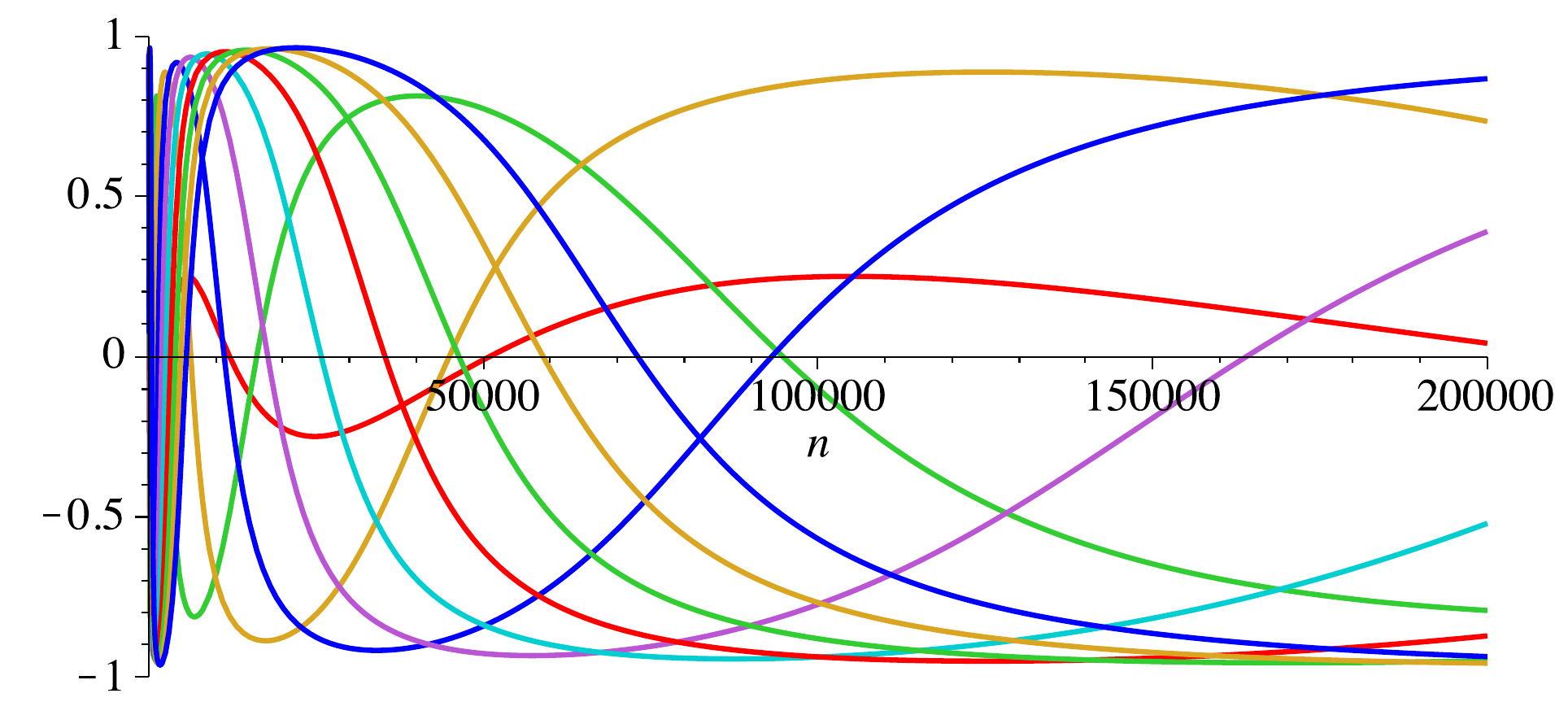}
\end{center}
\vspace*{-.3cm}
\caption{\emph{The periodic functions $F_\rho(2\pi t)$
for $m=27,\dots, 100$ (left) and $F_\rho(\beta\log n)$
for $m=27, 54, \dots, 270$ (right).}}
\label{fig-corr-coeff}
\end{figure}

One reason why the above result \eqref{rho-SN} may seem less or even
counter-intuitive is because of the seemingly strong dependence of
$N_n$ on $S_n$ in the recursive equations satisfied by both random
variables
\begin{align}\label{S-N}
    \begin{cases}
        S_n
        \stackrel{d}{=} S_{I_1}^{(1)}+\cdots+S_{I_m}^{(m)} +1,\\
        N_n
        \stackrel{d}{=} N_{I_1}^{(1)}+\cdots+N_{I_m}^{(m)} +
            S_{I_1}^{(1)}+\cdots+S_{I_m}^{(m)},
    \end{cases}
\end{align}
where the $(S_i^{(r)},N_i^{(r)})$'s are independent copies of
$(S_i,N_i)$, respectively, also independent of $(I_1,\ldots,I_m)$, and
\begin{align}\label{proba}
    {\mathbb P}(I_1=i_1,\dots, I_m=i_m)
    = \frac1{\binom{n}{m-1}},
\end{align}
when $i_1,\dots,i_m\ge0$ and $i_1+\cdots+i_m=n-m+1$. Intuitively, we
expect, from the above relations, that the node path length $N_n$
would have a strong correlation with $S_n$.

While one might ascribe this seemingly less intuitive result to the
possibly nonlinear dependence between $N_n$ and $S_n$, we enhance
such an uncorrelation by a stronger joint limit law for $(S_n, N_n)$
for $3\le m\le 26$, which further accents the asymptotic independence
between $N_n$ and $S_n$; for $m\ge27$, they are asymptotically
dependent and we will derive a precise characterization of their
joint asymptotic distributions. See Section~\ref{sec_bivariate} for a
more precise description of the joint asymptotic behaviors of
$(S_n,N_n)$ and $(S_n,K_n)$.

Let $\alpha$ denote the real part of the second largest zero (in real
parts) of the indicial equation $\Lambda(z)=0$, where
\begin{align}\label{indicial}
    \Lambda(z)
    = z(z+1)\cdots (z+m-2) - m!.
\end{align}
Then $\alpha<1$ for $m<14$ and $1<\alpha<\frac32$ for $14\le m\le
26$; see Table~\ref{tab-alpha}. Also $\alpha\to2$ as $m\to\infty$;
see \cite[Sec. 3.3]{mahmoud92} for more properties of $\alpha$.
\begin{table}[!ht]
\begin{center}
\begin{tabular}{|c||c|c|c|c|c|c|c|c|} \hline
$m$ & $3$ & $4$ & $5$ & $6$ & $7$ & $8$ & $9$ & $10$ \\ \hline
$\alpha$ & $-3$ & $-2.5$ & $-1.5$
& $-0.768$ & $-0.260$ & $0.101$
& $0.366$ & $0.568$ \\ \hline\hline
$m$ & $11$ & $12$ & $13$ & $14$ & $15$
& $16$ & $17$ & $18$ \\ \hline
$\alpha$ & $0.726$ & $0.852$ & \color{red}{$0.955$}
& \color{red}{$1.040$} & $1.112$ & $1.173$ & $1.226$ & $1.272$
\\ \hline\hline
$m$ & $19$ & $20$ & $21$ & $22$ & $23$
& $24$ & $25$ & $26$ \\ \hline
$\alpha$ & $1.313$ & $1.348$ & $1.380$ & $ 1.409$
& $1.435$ & $1.458$ & $1.479$ & \color{red}{$1.499$} \\ \hline
\end{tabular}
\end{center}
\vspace*{-.3cm}
\caption{\emph{Approximate numerical values of $\alpha=\alpha_m$
for $3\le m\le 26$.}}\label{tab-alpha}
\end{table}
The main reason that $\rho(S_n,N_n)\to0$ for $3\le m\le 26$ is
roughly that their covariance is of order $\max\{n\log
n,n^{\alpha}\}$ (see Theorem~\ref{rn_coas} below), while the standard
deviations for $S_n$ and $N_n$ are of orders $\sqrt{n}$ and $n$,
respectively. So that
\[
    \rho(S_n,N_n)
    = \begin{cases}
        O\left(n^{-\frac12}\log n\right),
            & \text{if } 3\le m\le 13;\\
        O\left(n^{-\frac32+\alpha}\right),
            & \text{if } 14\le m\le 26,
    \end{cases}
\]
which tends to zero in both cases. Briefly, \emph{the large quadratic
variance of $N_n$ is the major cause of the asymptotic independence
between $S_n$ and $N_n$ for $3\le m\le 26$}.

Such a change from being asymptotically independent to being
asymptotically dependent under a varying structural parameter is not
an exception. We will extend our study to fringe-balanced binary
search trees and quadtrees; a typical related instance states that:
\emph{the number of comparisons (or exchanges) used by the
median-of-$(2t+1)$ quicksort is asymptotically independent of the
number of partitioning stages when $0\le t\le 58$, but is
asymptotically dependent for $t\ge 59$.}

\section{$M$-ary search trees}

We briefly introduce $m$-ary search trees in this section and then
describe the random variables we are studying in this paper.

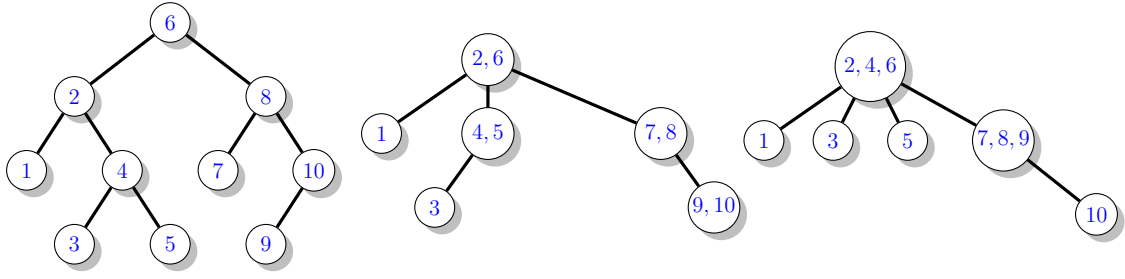
\begin{figure}[!ht]
\begin{center}
%%tikz start
\begin{tikzpicture}[scale=1,transform shape]
\node[] (P1) at (0,0) {
\begin{tikzpicture}[scale=0.7,transform shape,
every node/.style={inner sep=3pt,minimum size=0.75cm,text=blue,
draw,circle,fill=white,rounded corners=10pt,drop shadow},
]
\node []  (L1) at (0,0) {$6$};
\node [yshift=-1.4cm, xshift=-1.8cm]  (L21) at (L1) {$2$};
\node [yshift=-1.4cm, xshift=1.8cm]  (L22) at (L1) {$8$};
\node [yshift=-1.4cm, xshift=-0.9cm]  (L31) at (L21) {$1$};
\node [yshift=-1.4cm, xshift=0.9cm]  (L32) at (L21) {$4$};
\node [yshift=-1.4cm, xshift=-0.9cm]  (L33) at (L22) {$7$};
\node [yshift=-1.4cm, xshift=0.9cm]  (L34) at (L22) {$10$};
\node [yshift=-1.4cm, xshift=-0.9cm]  (L41) at (L32) {$3$};
\node [yshift=-1.4cm, xshift=0.9cm]  (L42) at (L32) {$5$};
\node [yshift=-1.4cm, xshift=-0.9cm]  (L43) at (L34) {$9$};

\draw[very thick] (L1) -- (L21);
\draw[very thick] (L1) -- (L22);
\draw[very thick] (L31) -- (L21);
\draw[very thick] (L32) -- (L21);
\draw[very thick] (L33) -- (L22);
\draw[very thick] (L34) -- (L22);
\draw[very thick] (L32) -- (L41);
\draw[very thick] (L32) -- (L42);
\draw[very thick] (L34) -- (L43);
\end{tikzpicture}
};

\node[] (P2) at (5,0) {
\begin{tikzpicture}[scale=0.7,
every node/.style={inner sep=3pt,minimum size=0.75cm,text=blue,
draw,fill=white,circle,drop shadow},
rect/.style={draw,scale=0.8, rectangle}
]
\node []  (L1) at (0,0) {$2,6$};
\node [yshift=-1.4cm, xshift=-2cm]  (L21) at (L1) {$1$};
\node [yshift=-1.4cm, xshift=0cm]  (L22) at (L1) {$4,5$};
\node [yshift=-1.4cm, xshift=3.25cm]  (L23) at (L1) {$7,8$};
\node [yshift=-1.4cm, xshift=-1cm]  (L31) at (L22) {$3$};
\node [inner sep=1pt,yshift=-1.4cm, xshift=1.0cm]
(L36) at (L23) {$9,10$};

\draw[very thick] (L1) -- (L21);
\draw[very thick] (L1) -- (L22);
\draw[very thick] (L1) -- (L23);
\draw[very thick] (L31) -- (L22);
\draw[very thick] (L36) -- (L23);
\end{tikzpicture}
};

\node[] (P3) at (10,0) {
\begin{tikzpicture}[scale=0.7,
every node/.style={inner sep=3pt,minimum size=0.75cm,text=blue,
draw,fill=white,circle,drop shadow},
rect/.style={draw,scale=0.8, rectangle}
]
\node []  (L1) at (0,0) {$2,4,6$};
\node [yshift=-1.4cm, xshift=-2cm]  (L21) at (L1) {$1$};
\node [yshift=-1.4cm, xshift=-0.7cm]  (L22) at (L1) {$3$};
\node [yshift=-1.4cm, xshift=0.7cm]  (L23) at (L1) {$5$};
\node [inner sep=1.25pt,yshift=-1.4cm, xshift=2.5cm]
(L24) at (L1) {$7,8,9$};
\node [yshift=-1.4cm, xshift=1.75cm]  (L34) at (L24) {$10$};

\draw[very thick] (L1) -- (L21);
\draw[very thick] (L1) -- (L22);
\draw[very thick] (L1) -- (L23);
\draw[very thick] (L1) -- (L24);
\draw[very thick] (L34) -- (L24);
\end{tikzpicture}
};

\end{tikzpicture}
\end{center}
\vspace*{-.3cm}
\caption{\emph{Three $m$-ary search trees for the sequence
$\{6,2,4,8,7,1,5,3,10,9\}$: $m=2$ (left), $m=3$ (middle), and
$m=4$ (right).}}\label{fig-3trees}
\end{figure}

An \emph{$m$-ary tree} is either empty or comprises of a single node
called the root, together with an ordered $m$-tuple of subtrees, each
of which is, by definition, an $m$-ary tree. Given a sequence of
numbers, say $\{x_1,\dots,x_n\}$, we construct an $m$-ary search tree
by the following procedure, $m\ge2$. If $1\le n<m$, then all keys are
stored in the root. If $n\ge m$ the first $m-1$ keys are sorted and
stored in the root, the remaining keys are directed to the $m$
subtrees, each corresponding to one of the $m$ intervals formed by
the $m-1$ sorted keys in the root node; see Figure~\ref{fig-3trees}
for an illustration (the rectangular nodes denote yet empty subtrees
of full nodes). If the $m-1$ numbers in the root are $x_{j_1}<
\cdots<x_{j_{m-1}}$, then the keys directed to the $i$th subtree all
have their values lying between $x_{j_{i-1}}$ and $x_{j_i}$, where
$x_{j_0}:=0$ and $x_{j_m}:=n+1$. All subtrees are themselves $m$-ary
search trees by definition. For more details, see Mahmoud
\cite{mahmoud92}.

While the practical usefulness of $m$-ary search trees is largely
overshadowed by their balanced counterparts such as $B$-trees, they
have been a source of many interesting phenomena, which are to some
extent universal. The study of $m$-ary search trees is thus of
fundamental and prototypical value. Furthermore, the close connection
between $m$-ary search trees and generalized quicksort adds an extra
dimension to the richness of diverse variations and their asymptotic
behaviors.

\subsection{Space requirement and total path lengths}

Assume that the input sequence $\{x_1,\dots,x_n\}$ is a random
permutation, where all $n!$ permutations are equally likely. The
resulting $m$-ary search tree constructed from the given sequence is
then called a random $m$-ary search tree. The major shape parameters
of particular algorithmic interest include the depth, the height, the
space requirement, the total path length, and the profile; see
\cite{drmota08,mahmoud92} for more information. We are concerned in
this paper with the following three random variables.

\begin{itemize}

\item $S_n$ (space requirement): the total number of nodes used to
store the input; the three trees in Figure~\ref{fig-3trees} have
$S_{10}$ equal to $10, 6, 6$, respectively. If $m=2$, then $S_n\equiv
n$; if $m\ge3$, we can compute $S_n$ recursively by $S_0=0$, and
\begin{align}\label{dis-sn}
    S_n
    \stackrel{d}{=}
    \begin{cases}
    1, & \text{if }1\le n<m,\\
    S_{I_1}^{(1)}
    +\cdots + S_{I_m}^{(m)} +1, & \text{if } n\ge m,
    \end{cases}
\end{align}
where the $S_i^{(r)}$'s are independent copies of $S_i$, $1\le r\le
m$, $0\le i\le n-m+1$, and independent of $(I_1,\ldots,I_m)$ defined
in \eqref{proba}.

\item $K_n$ (key path length, KPL): the sum of the distance between
the root and each key; for the trees in Figure~\ref{fig-3trees},
$K_{10}=\{19, 11,8\}$, respectively. For $m\ge2$, $K_n$ satisfies the
recurrence
\begin{align}\label{dis-xn}
    K_n
    \stackrel{d}{=}
    \begin{cases}
    0, & \text{if }n<m,\\
    K_{I_1}^{(1)}
    +\cdots + K_{I_m}^{(m)} +n-m+1, & \text{if } n\ge m,
    \end{cases}
\end{align}
where the $K_i^{(r)}$'s are independent copies of $K_i$, $1\le r\le
m, 0\le i\le n-m+1$, independent of $(I_1,\ldots,I_m)$.

\item $N_n$ (node path length, NPL): the sum of the distance between
the root and each node; so that $N_{10}=\{19, 7, 6\}$ for the three
trees in Figure~\ref{fig-3trees}. Obviously, $N_n=K_n$ when $m=2$.
When $m\ge3$,
\begin{align}\label{dis-nn}
    N_n
    \stackrel{d}{=}
    \begin{cases}
    0, & \text{if }n<m,\\
    N_{I_1}^{(1)}
    +\cdots + N_{I_m}^{(m)} + S_{I_1}^{(1)}
    +\cdots + S_{I_m}^{(m)}, & \text{if } n\ge m,
    \end{cases}
\end{align}
where the $(N_i^{(r)},S_i^{(r)})$'s are independent copies of
$(N_i,S_i)$, $1\le r\le m, 0\le i\le n-m+1$, independent of
$(I_1,\ldots,I_m)$.

\end{itemize}

While the first two random variables have been widely studied in the
literature, NPL was only considered previously in
\cite{broutin12,holmgren11} in connection with the process of cutting
trees. In addition to this, our interest was to understand the extent
to which the asymptotic independence for small $m$ between $S_n$ and
$K_n$ subsists when the ``toll function'' changes from a linear
function to a function that is random and may depend on $S_n$.

\subsection{A summary of known results}\label{known-results}

Let $H_m := \sum_{1\le j\le m}j^{-1}$. Knuth \cite[\S 6.2.4]{knuth98}
was the first to show that
\[
    \mathbb{E}(S_n)
    \sim \phi n,\quad\text{where}\quad
	\phi := \frac1{2(H_m-1)},
\]
(see also \cite{baeza-yates87}). Here $\phi$ denotes the ``occupancy
constant", which will appear all over our analysis. Mahmoud and
Pittel \cite{mahmoud89} improved the result and derived an identity
for $\mathbb{E}(S_n)$, which implies in particular that
\[
    \mathbb{E}(S_n)
    = \phi (n+1) - \frac1{m-1}
    + O\left(n^{\alpha-1}\right),
\]
where $\alpha$ has the same meaning as in Introduction; see
\eqref{indicial}. They also discovered and proved the surprising
result for the variance
\begin{align*}%\label{Sn-var}
    \mathbb{V}(S_n)
    \sim \begin{cases}
        C_S n,
            & \text{if } 3\le m\le 26;\\
        F_1(\beta\log n) n^{2\alpha-2},
            & \text{if } m\ge27,
    \end{cases}
\end{align*}
where $C_S$ is a constant depending on $m$, $F_1$ is a $\pi$-periodic
function given in \eqref{F1z}, $\alpha+i\beta$ is the second largest
zero (in real part) with $\beta>0$ of the equation $\Lambda(z)=0$
(see \eqref{indicial}), and $2\alpha-2>1$ for $m\ge27$. See also
\cite{dean02,janson08,majumdar05} for a closely related fragmentation
model with the same asymptotic behavior. A central limit theorem for
$S_n$ was then proved for $3\le m\le 26$ in \cite{lew94,mahmoud89};
see also \cite{mahmoud92} for more details. Their approach is based
on an inductive approximation argument.

By the method of moments, two authors of this paper re-proved in
\cite{chern01} the central limit theorem for $S_n$ when $3\le m\le
26$; the same approach was also used to establish the nonexistence of
a limit law for $S_n$ due to inherent oscillations. Moreover, the
convergence rates to the normal distribution were characterized in
\cite{hwang03} by a refined method of moments, which undergo further
change of behaviors.

Then several different approaches were developed in the literature
for a deeper understanding of the ``phase change" at $m=26$; these
include martingale \cite{chauvin04}, renewal theory \cite{janson08},
urn models \cite{janson04,mailler14}, contraction method
\cite{fill04,neininger04}, method of moments \cite{hwang03},
statistical physics \cite{dean02,majumdar05}, etc.

On the other hand, the KPL for general $m\ge 2$ was first studied by
Mahmoud \cite{mahmoud86} and he proved
\[
    \mathbb{E}(K_n)
    = 2\phi n\log n +c_1 n + o(n),
\]
for some explicitly computable constant $c_1$; see \eqref{c1}. The
variance was computed in \cite[\S 3.5]{mahmoud92} and satisfies
($H_m^{(2)}:= \sum_{1\le j\le m}j^{-2}$)
\begin{align}\label{VKn}
    \mathbb{V}(K_n) \sim C_K n^2,
    \quad\text{where}\quad
    C_K = 4\phi^2\left(
    \tfrac{(m+1)H_m^{(2)}-2}{m-1}
    -\tfrac{\pi^2}6\right).
\end{align}
The corresponding limit law was characterized in \cite{neininger99}
by the contraction method
\begin{align}\label{Kn-ll}
    \frac{K_n - \mathbb{E}(K_n)}{n}
    \stackrel{d}{\longrightarrow} K,
\end{align}
where $K$ is given by the recursive distributional equation
(\ref{rn1212}); see also \cite{broutin12,munsonius11} for a general
framework.

For NPL $N_n$, Broutin and Holmgren \cite{broutin12} proved that
\[
    \mathbb{E}(N_n) = 2\phi^2 n\log n + c_2n + o(n),
\]
for some constant $c_2$ (for which no numerical value was provided);
a series expression of $c_2$ is given in \cite[p.\ 156]{holmgren11}.
We will give an alternative proof of this result below with tools
from \cite{chern01,fill05a}. Our approach makes the computation of
$c_2$ feasible (although its exact value is not needed); see
\eqref{c2}.

It should be mentioned that there is a large literature on $K_n$ when
$m=2$ because it is identical to the comparison cost used by
quicksort. Many fine results were obtained; see, for example, the
recent papers \cite{bindjeme12,fill13,fuchs14a,
kabluchko14,neininger14,sulzbach14} and the references therein for
more information.

\subsection{Covariance, correlation, dependence and phase changes}

We state in this section our results for the covariance and
correlation between the space requirement and the total path lengths
(KPL and NPL). The proofs and the tools needed will be given in the
next sections.

Unlike the space requirement $S_n$ whose variance changes its
asymptotic behavior for $m\ge27$, the covariance $\Cov(S_n,\SK_n)$
changes its asymptotic behavior at $m=14$.

\begin{thm} \label{cov-snxn} The covariance between $S_n$ and $K_n$
satisfies
\[
    \Cov(S_n,\SK_n)
    \sim \begin{cases}
        C_Rn,
            &\text{if}\ 3\le m\le 13;\\
        F_2\left(\beta\log n\right)n^{\alpha},
            &\text{if}\ m\ge 14;
    \end{cases}
\]
where $C_R$ is a suitable constant and $F_2(z)$ is a $2\pi$-periodic
function given in \eqref{F2z} below.
\end{thm}

This result has the following consequence.
\begin{cor}
The correlation coefficient between $S_n$ and $K_n$ satisfies
\[
    \rho(S_n,\SK_n) \begin{cases}
        \rightarrow 0, & \text{if } 3\le m\le 26;\\
        \sim \displaystyle
        \frac{F_2\left(\beta\log n\right)}
        {\sqrt{C_KF_1(\beta\log n)}}, & \text{if }
        m\ge27,
    \end{cases}
\]
where $C_K>0$ is given in \eqref{VKn}.
\end{cor}
See Figure~\ref{fig-corr-coeff} for two different plots for the
periodic functions when $m\ge27$.

The same consideration extends easily to clarify the correlation
between space requirement and NPL.
\begin{thm}\label{rn_coas}
The covariance between $S_n$ and $N_n$ satisfies
\[
    \Cov(S_n,N_n)
    \sim
    \begin{cases}
        2\phi C_Sn\log n,
            &\text{if}\ 3\le m\le 13;\\
        \phi F_2\left(\beta\log n\right)n^{\alpha},
            &\text{if}\ m\ge 14,
    \end{cases}
\]
where $C_S$ is as in Section \ref{known-results}. Moreover, the
variance of $N_n$ satisfies
\[
    \V(\SN_n)
    \sim \phi^2C_K n^2.
\]
\end{thm}
Notice the appearance of an extra $\log n$ factor when $3\le m\le
13$, which reflects the additional random effect introduced by the
toll function in \eqref{dis-nn}. These estimates imply the following
consequence.

\begin{cor} The correlation coefficient $\rho(S_n,\SN_n)$ satisfies
\[
    \rho(S_n,N_n) \begin{cases}
        \rightarrow 0, & \text{if } 3\le m\le 26;\\
        \sim \rho(S_n,\SK_n) \sim \displaystyle
        \frac{F_2\left(\beta\log n\right)}
        {\sqrt{C_K F_1(\beta\log n)}}, & \text{if }
        m\ge27.
    \end{cases}
\]
\end{cor}

The last relation suggests considering the correlation between
$K_n$ and $N_n$.
\begin{cor} The random variable $K_n$ is asymptotically linearly
correlated to $N_n$
\[
    \rho(K_n,N_n)\to 1.
\]
\end{cor}
Indeed, we will show that
\[
\| N_n-\phi K_n-({\mathbb E}(N_n-\phi K_n))\|_2=o(n)
\]
which then by Slutsky's theorem implies that
\[
    \left(\frac{K_n-{\mathbb E}(K_n)}{n},\frac{N_n -\mathbb{E}(N_n)}{n}\right)
    \stackrel{d}{\longrightarrow} (K,\phi K);
\]
see \eqref{Kn-ll}, Section \ref{univariate} and \ref{app-univariate}.

These results will be proved by working out the asymptotics of the
corresponding recurrence relations, which all have the same form
\[
    a_n
    =m\sum_{0\le j\le n-m+1}\pi_{n,j}a_j+b_n,
    \qquad (n\ge m-1),
\]
where
\[
    \pi_{n,j}
    =\frac{\binom{n-1-j}{m-2}}{\binom{n}{m-1}}
    \qquad(0\le j\le n-m+1)
\]
is a probability distribution, and $\{b_n\}$ is a given sequence
(referred to as the toll-function). For that asymptotic purpose, our
key tools will rely on the \emph{asymptotic transfer techniques} (see
\cite{chern01,fill05a}), which provide a direct asymptotic
translation from the asymptotic behaviors of $b_n$ to those of $a_n$.
The remaining analysis will then consist of simplifying some multiple
Dirichlet's integrals.

Since Pearson's product-moment correlation coefficient $\rho$ is
known to be poor in measuring nonlinear dependence between two random
variables, we go further by considering the joint limit laws for
$(S_n,K_n)$ and $(S_n,N_n)$, which exhibit a change of behavior
depending on whether $3\le m\le 26$ (convergent case) or $m\ge27$
(periodic case): they are asymptotically independent in the former
case but dependent in the latter.

\begin{thm}\label{thm13} Assume $3\le m\le 26$. Let $(X_n)_n \in
\{(K_n)_n,(N_n)_n\}$ and $Q_n=(X_n,S_n)$ denote the vector of KPL or
NPL and the space requirement used by a random $m$-ary search tree.
Then the convergence in distribution holds:
\begin{align}\label{thmnc}
    \Cov(Q_n)^{-1/2}(Q_n-\E[Q_n])
    \stackrel{d}{\longrightarrow} (X,\mathscr{N}),
\end{align}
where $\mathscr{N}$ has the standard normal distribution and the
limit law $(X,\mathscr{N})$ is described in Lemma~\ref{fpm23};
moreover, $X$ and $\mathscr{N}$ are independent.
\end{thm}

\begin{thm}\label{thm12}
Assume $m\ge 27$. Let $(X_n)_n \in \{(K_n)_n,(N_n)_n\}$ and
\begin{align*}%\label{norma}
    Y_n
    :=\left(\frac{X_n-\E[X_n]}{\iota_Xn},
        \frac{S_n-\phi n}{n^{\alpha-1}}\right)
\end{align*}
with $\iota_X=1$ for $(X_n)_n=(N_n)_n$ and $\iota_X=\phi^{-1}$ for
$(X_n)_n=(K_n)_n$. Then we have
\begin{align*}
    \ell_2(Y_n, (X,\Re(n^{i\beta} \Lambda)))
    \to 0,
\end{align*}
where $\beta$ is as in Section \ref{known-results} and $(X,\Lambda)$
is a random vector whose distribution is specified as the unique
fixed point solution appearing in Lemma~\ref{fplem1} for the choice
$\gamma=(0,\theta)$ ($\theta$ being defined below in (\ref{exp_mu})).
\end{thm}

See Section~\ref{sec_bivariate} for a more precise formulation. The
proof is based on the \emph{contraction method} (see
\cite{neininger01}) where we use the above moment asymptotics as
input and combine well-known estimates within the minimal
$L_2$-metric for the convergent case (as in \cite{rosler01}), and
those with estimates for the periodic case (as in \cite{fill04}).
Similar proof techniques related to periodic distributional behaviors
are also applied in \cite[Theorem 1.3(iii)]{janson08} and
\cite[Theorem 6.10]{knape14}. If one is only interested in the
asymptotic (univariate) distribution of the NPL $N_n$ (the case of
the KPL being known before), there are more direct proofs which we
also discuss in Sections~\ref{univariate} and \ref{app-univariate}.

Our study of the dependence of random variables on random $m$-ary
search trees can be extended in at least two directions by the same
methods used in this paper, namely, asymptotic transfer techniques
and the contraction method. \begin{itemize}

\item \emph{Extension to more general linear and $n\log n$ shape
measures}: That the asymptotic covariance undergoes a phase change
after $m=13$ and the asymptotic correlation undergoes a phase change
after $m=26$ is not restricted to the space requirement and KPL or
NPL. Indeed, we can replace the space requirement by many other
linear shape measures such as the number of leaves, the number of
nodes of a specified type, the number of occurrences of a fixed
pattern, etc. (see \cite{chern01} for more examples), and KPL or NPL
by other shape measures with mean of order $n\log n$ such as summing
over the root-node or root-key distance for certain specified nodes
or patterns and weighted path length.

\item \emph{Extension to other random trees of logarithmic height}:
the same change of asymptotic behaviors from being independent to
being dependent under a varying structural parameter also occurs in
other classes of random log-trees; we content ourselves with the
brief discussion of two classes of random trees:
\emph{fringe-balanced binary search trees} and \emph{quadtrees}. The
behaviors will be however very different for the classes of trees
where the underlying distribution of the subtree sizes are dictated
by a binomial distribution, which will be examined elsewhere; see a
companion paper \cite{fuchs16} for more information.

\end{itemize}

This paper is organized as follows. We prove in the next section our
results for the covariances and the correlations. These results are
then used to study the bivariate distributional asymptotics in
Section~\ref{sec_bivariate} by the multivariate contraction method
(see \cite{neininger01}). Finally, in Section~\ref{sec_ext}, we
discuss the dependence and phase changes in fringe-balanced binary
search trees and in quadtrees, where for the former, we study the
joint behavior of the size and total path length, while for the
latter (since the size is a constant) we consider the joint behavior
of the number of leaves and total path length. Also we include a
brief discussion for extending the study and results to other shape
parameters in Section~\ref{sec_ext}.

\section{Correlation between space requirement and path lengths}
\label{sec_spl}

We prove in this section Theorems~\ref{cov-snxn} and \ref{rn_coas}
for the covariances $\text{Cov}(S_n,K_n)$ and $\text{Cov}(S_n,N_n)$,
respectively.

\subsection{Preliminaries and recurrences}
\label{sec_pre}

We collect here the notations to be used in the proofs. Let $m\ge2$
be a fixed integer. For $n\ge m$, denote by $I^{(n)}
=(I^{(n)}_1,\ldots,I^{(n)}_m)$ the vector of the number of keys
inserted in the $m$ ordered subtrees of the root in a random $m$-ary
search tree with $n$ keys. When the dependence on $n$ is obvious, we
write simply $(I_1,\ldots,I_m)$. Generate independently $n$ uniform
random variables $U_1,\ldots,U_n$ on $[0,1]$. Store the first $m-1$
elements $U_1,\ldots,U_{m-1}$ in the root-node of the tree. Then they
decompose the unit interval $[0,1]$ into spacings of lengths
$V_1,\ldots,V_m$, where $V_j=U_{(j)}-U_{(j-1)}$ for $j=1,\ldots,m$
with $U_{(0)}:=0, U_{(m)}:=1$ and $U_{(j)}$ for $j=1,\ldots,m-1$ are
the order statistics of $U_1,\ldots,U_{m-1}$. The uniform permutation
model implies, that, conditional on $U_1,\ldots,U_{m-1}$, the vector
$I^{(n)}$ has the multinomial distribution with success probabilities
$V_1,\ldots,V_m$, namely, we have
\[
    (I_1,\ldots,I_m)
    \stackrel{\mathrm{d}}{=}M(n-m+1;V_1,\ldots,V_m).
\]
In particular, we have the convergence
\begin{align}\label{lim_ij}
    \frac{I_r}{n}
    \longrightarrow V_r,
\end{align}
for all $r=1,\ldots,m$, where the convergence is in $L_p$ for all
$1\le p<\infty$. Note that we also have \eqref{proba} for all
$m$-tuples $i_1,\ldots,i_m\ge 0$ with $i_1+\cdots+i_m=n-m+1$ and all
$n\ge m$.

For each of the subtrees, the randomness (uniformity) is preserved;
more precisely, conditional on the number of keys inserted in a
subtree, each subtree has the same distribution as a random $m$-ary
search tree of that number of keys in the uniform model. Moreover,
conditional on $(I_1,\ldots,I_m)$, the subtrees are independent. This
can be seen by switching back to the ranks $\{1,\ldots,n\}$ of the
input elements, and then by checking that a uniform random
permutation yields independent permutations on the respective ranges.
This recursive structure of the random $m$-ary search tree implies
the recursive relations for $S_n, K_n$ and $N_n$ given in
\eqref{dis-sn}--\eqref{dis-nn}, where the summands appearing on the
right-hand sides, namely, $S_j^{(1)},\ldots,S_j^{(m)}$ and
$\SK_j^{(1)},\ldots, \SK_{j}^{(m)}$ and
$N_j^{(1)},\ldots,N_{j}^{(m)}$ have the same distributions as $S_j$
and $\SK_j$ and $N_j$, respectively. Furthermore, the triples
$\left(
        \bigl(S_j^{(r)}\bigr)_{0\le j \le n-m+1},
        \bigl(\SK_j^{(r)}\bigr)_{0\le j \le n-m+1},
        \bigl(N_j^{(r)}\bigr)_{0\le j \le n-m+1}
    \right)
$
are independent for $r=1,\ldots,m$ and independent of
$(I_1,\ldots,I_m)$. Finally, the recursive structure of the $m$-ary
search tree implies recurrences satisfied by their joint
distributions. In particular, the pair $Q_n:=(N_n,S_n)$ satisfies the
recurrence
\begin{align}\label{rec1}
    \left(Q_n\right)^{\mathrm{t}}
    \stackrel{d}{=}\sum_{1\le r\le m}  \Bigl[
    \begin{array}{cc}
        1 & 1\\
        0& 1
    \end{array}\Bigr] \left(Q^{(r)}_{I_r}\right)^{\mathrm{t}}
    + \binom{0}{1}, \qquad (n\ge m),
\end{align}
where, as in (\ref{dis-sn})--(\ref{dis-nn}), the $Q^{(r)}_j$'s are
distributed as $Q_j$ for all $1\le r \le m$ and $0\le j\le n-m+1$,
and the $\bigl(Q^{(r)}_j\bigr)_{0\le j\le n-m+1}$ are independent for
$r=1,\ldots,m$ and independent of $(I_1,\ldots,I_n)$. The recurrence
satisfied by the pair $Z_n:=(K_n,S_n)$ is
\begin{align}\label{rec1b}
    \left(Z_n\right)^{\mathrm{t}}
    \stackrel{d}{=}\sum_{1\le r\le m}  \Bigl[
    \begin{array}{cc}
        1 & 0\\ 0& 1
    \end{array}\Bigr] \left(Z^{(r)}_{I_r}\right)^{\mathrm{t}}
    +\binom{n-m+1}{1} , \qquad (n\ge m),
\end{align}
with conditions on independence and identical distributions similar
to (\ref{rec1}).

\subsection{Asymptotic transfer and Dirichlet integrals}

Starting from the distributional recurrences (\ref{dis-sn}) and
(\ref{dis-xn}), we see that all centered and non-centered moments
satisfy the same recurrence of the following type
\begin{equation}\label{und-rec}
    a_n
    =m\sum_{0\le j\le n-m+1}\pi_{n,j}a_j+b_n,
    \qquad \pi_{n,j}
    =\frac{\binom{n-1-j}{m-2}}{\binom{n}{m-1}},
\end{equation}
for $n\ge m-1$, where $\{b_n\}_{n\ge m-1}$ is a given sequence. The
asymptotics of $a_n$ can be systematically characterized by that of
$b_n$ through the use of the following transfer techniques; see
Proposition 7 in \cite{chern01} and Theorem 2.4 in \cite{fill05a} for
details.

\begin{pro}\label{asymp-trans} Assume that $a_n$ satisfies
(\ref{und-rec}) with finite initial conditions $a_0,\dots,a_{m-2}$.
Define $b_n:=a_n$ for $0\le n \le m-2$.
\begin{itemize}

\item[(i)] Assume $b_n=c(n+1)+t_n$, where $c\in\mathbb{C}$. Then the
conditions
\begin{align*}%\label{bn-iff}
    t_n
    =o(n)\quad\text{and}\quad
        \left\vert\sum_{n\ge1}t_nn^{-2}\right\vert<\infty
\end{align*}
are both necessary and sufficient for
\[
    a_n = 2c\phi nH_n+c'n+o(n),
\]
where
\[
    c' = 2\phi \sum_{j\ge 0}\frac{t_j}{(j+1)(j+2)}
    +\frac{c}2-2c\phi+2c(H_m^{(2)}-1)\phi^2;
\]
\item[(ii)] if $b_n\sim cn^{v}$, where $v>1$, then
\[
    a_n
    \sim\frac{c}{1-\frac{m!\Gamma(v+1)}{\Gamma(v+m)}}n^v.
\]
\end{itemize}
\end{pro}
In particular, when $c=0$ in $(i)$, then we see that $a_n$ is
asymptotically linear
\[
    \frac{a_n}{n} \sim 2\phi\sum_{j\ge0}\frac{b_j}{(j+1)(j+2)}
    \quad \text{iff}\quad
    b_n
    =o(n)\quad\text{and}\quad
    \left\vert\sum_{n\ge1}b_nn^{-2}\right\vert<\infty.
\]

We will be dealing with Dirichlet integrals of the following type
\[
    I(u,v)
    :=\int_{\substack{x_1+\cdots+x_m=1\\
    0\le x_1,\ldots,x_m\le 1}}\left(\sum_{1\le l\le m}
    x_l^{u-1}\right)\left(\sum_{1\le r\le m}x_r^{v-1}\right)
    \mathrm{d}\mathbf{x},
    \qquad(\Re(u), \Re(v)>0).
\]
Here $\mathrm{d}\mathbf{x}$ is an abbreviation for
$\mathrm{d}x_1\cdots \mathrm{d}x_{m-1}$. Such integrals have a
closed-form expression.
\begin{lem}\label{Iab}
For $m\ge 2$ and $\Re(u), \Re(v)>0$,
\begin{align}\label{diri-int1}
    I(u,v) = \frac{m\Gamma(u+v-1)+m(m-1)\Gamma(u)\Gamma(v)}
    {\Gamma(u+v+m-2)}.
\end{align}
\end{lem}
\pf First, the claim is easily proved for $m=2$. Assume $m\ge 3$.
Then, by symmetry,
\begin{align*}
    I(u,v)
    &=\int_{\substack{x_1+\cdots+x_m=1\\
        0\le x_1,\ldots,x_m\le 1}}
        \left(mx_1^{u+v-2}+m(m-1)x_1^{u-1}x_2^{v-1}\right)
        \mathrm{d}\mathbf{x}\\
    &=\frac{m}{(m-2)!}
        \int_{0}^{1}x_1^{u+v-2}(1-x_1)^{m-2}\mathrm{d}x_1\\
    &\qquad\qquad+\frac{m(m-1)}{(m-3)!}
        \int_{0}^{1}\int_{0}^{1-x_1}
        x_1^{u-1}x_2^{v-1}(1-x_1-x_2)^{m-3}
        \mathrm{d}x_2\mathrm{d}x_1\\
    &=\frac{m\Gamma(u+v-1)}{\Gamma(u+v+m-2)}
        +\frac{m(m-1)\Gamma(u)\Gamma(v)}{\Gamma(u+v+m-2)},
\end{align*}
which leads to \eqref{diri-int1}.\qed

The following two identities will be needed below.
\begin{align}\label{diri-int2}
\begin{split}
    &\int_{\substack{x_1+\cdots+x_m=1\\
    0\le x_1,\ldots,x_m\le 1}}\left(\sum_{1\le l\le m}
    x_l^{u-1}\right)\left(\sum_{1\le r\le m}x_r\log x_r\right)
    \mathrm{d}\mathbf{x}\\
    &\qquad = \frac{\partial}{\partial v}\,I(u,v)\Biggr|_{v=2}\\
    &\qquad = \frac{m\Gamma(u)}{\Gamma(m+u)}
    \left(u\psi(u+1)+(m-1)(1-\gamma)-(m+u-1)\psi(m+u)\right),
\end{split}
\end{align}
where $\psi$ is the digamma function and $\gamma$ is Euler's
constant. Similarly,
\begin{align}\label{diri-int3}
\begin{split}
    \int_{\substack{x_1+\cdots+x_m=1\\
    0\le x_1,\ldots,x_m\le 1}}
    \left(\sum_{1\le r\le m}x_r\log x_r\right)^2
    \mathrm{d}\mathbf{x}
    &= \frac{\partial^2}{\partial u\partial v}\,I(u,v)
    \Biggr|_{u=v=2}\\
    &=H_m^{(2)} +\frac4{\phi^2}-\frac2{m+1}
    -\frac{(m-1)\pi^2}{6(m+1)}.
\end{split}
\end{align}

\subsection{Correlation between the space requirement and KPL}
\label{sec_corrmary}

We are now ready to prove Theorem~\ref{cov-snxn}.

\paragraph{Expected values of $S_n$ and $K_n$.} For convenience, let
$\mu_n := \E(S_n)$ and $\kappa_n := \E(K_n)$. Then, by \eqref{dis-sn}
and \eqref{dis-xn}, for $n\ge m-1$
\begin{align*}
    \mu_n &= m\sum_{0\le j\le n-m+1}\pi_{n,j} \mu_j + 1,\\
    \kappa_n &= m\sum_{0\le j\le n-m+1}\pi_{n,j} \kappa_j + n-m+1,
\end{align*}
with the initial conditions $\mu_0=\kappa_n=0$ for $0\le n\le m-2$
and $\mu_n = 1$ for $1\le n\le m-2$.

By applying Proposition~\ref{asymp-trans}(i), we obtain
\begin{align}\label{exp-nun}
    \mu_n \sim \phi n ,\quad\text{and} \quad
    \kappa_n = 2\phi n\log n+c_1n + o(n),
\end{align}
for some constant $c_1$ whose value matters less; see \eqref{c1}
below. The latter approximation is sufficient for all our purposes,
but the former is not and we need the following stronger expansion
(see \cite{chern01,mahmoud89,mahmoud92})
\begin{align}\label{exp-mun}
    \mu_n = \phi (n+1) - \frac1{m-1} +
    \sum_{2\le k\le 3}\frac{A_k}{\Gamma(\lambda_k)}\,
    n^{\lambda_k-1} + o(n^{\alpha-1}),
\end{align}
where $\lambda_2=\alpha+i\beta$ and $\lambda_3 := \alpha-i\beta$ and
\begin{align}\label{Ak}
    A_k
    =\frac{1}{\lambda_k(\lambda_k-1)
        \sum_{0\le j\le m-2}\frac1{j+\lambda_k}}.
\end{align}
Note that for $3\le m\le 13$ the constant term $-\frac1{m-1}$
(together with $\phi$) is the second-order term on the right-hand
side of \eqref{exp-mun}, while for larger $m$, it is absorbed in the
$o$-term.

On the other hand, although the explicit expression of $c_1$ is not
needed in this paper, we provide its expression here since the
known ones (see \cite{mahmoud86,mahmoud92}) are less explicit and
it can be easily obtained from Proposition~\ref{asymp-trans}:
\begin{align}\label{c1}
	c_1=-\tfrac12-4\phi+2\phi^2(H_m^{(2)}-1)+\gamma.
\end{align}

\paragraph{Variance and covariance.} To compute the asymptotics of
the covariance, we first derive the corresponding recurrences and
then apply Proposition~\ref{asymp-trans} of asymptotic transfer.

First, let $\bar{S}_n=S_n-\mu_n$ and $\bar{\SK}_n=\SK_n-\kappa_n$. We
consider the moment-generating function
\[
    \bar{P}_n(u,v)
    :=\E\left(e^{\bar{S}_nu+\bar{\SK}_nv}\right).
\]
Then, using (\ref{dis-sn}) and (\ref{dis-xn}), we obtain for $n\ge
m-1$
\begin{equation}\label{joint-snxn}
    \bar{P}_n(u,v)
    =\frac{1}{\binom{n}{m-1}}
        \sum_{\mathbf{j}}P_{j_1}(u,v)\cdots P_{j_m}(u,v)
        e^{\Delta_{\mathbf{j}}u+\nabla_{\mathbf{j}}v}
\end{equation}
with the initial conditions $\bar{P}_n(u,v)=1$ for $0\le n\le m-2$.
Here, $\mathbf{j}=(j_1,\ldots,j_m)$ is a vector with $j_1,\ldots,
j_m\ge 0$ and $j_1+\cdots+j_m=n-m+1$ (we use this notation
throughout),
\begin{equation}\label{def-dn}
    \Delta_{\mathbf{j}}
    =1-\mu_n+\sum_{1\le l\le m}\mu_{j_l},\quad\text{and}\quad
    \nabla_{\mathbf{j}}
    =n-m+1-\kappa_n+\sum_{1\le l\le m}\kappa_{j_l}.
\end{equation}

Define
\[
    V_n^{[S]}
    =\V(S_n),\qquad
    V_n^{[SK]}
    =\Cov(S_n,\SK_n),\qquad
    V_n^{[K]}
    =\V(\SK_n).
\]
Then, by taking derivatives in (\ref{joint-snxn}), we obtain
\begin{align*}
    V_n^{[X]}
    =m\sum_{0\le j\le n-m+1}\pi_{n,j}V_j^{[X]}+b_n^{[X]},
    \quad (X\in\{S,SK,K\}),
\end{align*}
where
\begin{align*}
    b_n^{[S]}
    =\frac{1}{\binom{n}{m-1}}
        \sum_{\mathbf{j}}\Delta_{\mathbf{j}}^2,\quad%\\
    b_n^{[SK]}
    =\frac{1}{\binom{n}{m-1}}
        \sum_{\mathbf{j}}\Delta_{\mathbf{j}}
        \nabla_{\mathbf{j}},\quad\text{and}\quad%\\
    b_n^{[K]}
    =\frac{1}{\binom{n}{m-1}}
        \sum_{\mathbf{j}}\nabla_{\mathbf{j}}^2.
\end{align*}

We first derive uniform asymptotic approximations for
$\Delta_{\mathbf{j}}$ and $\nabla_{\mathbf{j}}$.

\begin{lem}\label{est-dn}
Uniformly in $\mathbf{j}$,
\[
    \Delta_{\mathbf{j}}
    =\sum_{2\le k\le 3}\frac{A_k}
        {\Gamma(\lambda_k)}\,n^{\lambda_k-1}
        \left(-1+ \sum_{1\le r\le m}
        \left(\frac{j_r}{n}\right)^{\lambda_k-1}\right)
        +o(n^{\alpha-1}),
\]
and
\[
    \nabla_{\mathbf{j}}
    =n\left(1+2\phi \sum_{1\le r\le m}
        \frac{j_r}{n}\log\frac{j_r}{n}\right)+o(n).
\]
\end{lem}
\pf This follows from substituting the asymptotic approximations
(\ref{exp-nun}) and (\ref{exp-mun}) into (\ref{def-dn}), and standard
manipulations. \qed

\paragraph{Asymptotics of $V_n^{[S]}$.} Although the asymptotic
behaviors of the variance of $S_n$ have been computed before, we
re-derive them here by a different approach, which is easily amended
for the calculation of other variances and covariances.

Consider first $3\le m\le 26$. Then $\alpha<3/2$. Moreover, from
Lemma \ref{est-dn},
\[
    b_n^{[S]}
    =O(n^{2\alpha-2}) = O(n^{1-\ve}),
\]
for some $0<\ve<0.00171$. Consequently, by applying
Proposition~\ref{asymp-trans}(i),
\[
    V_n^{[S]}\sim C_Sn,
\]
for some constant $C_S$; see \cite{chern01} for a more explicit
expression and the proof that $C_S>0$.

On other hand, if $m\ge 27$, since $\alpha>3/2$, we then have, by
Lemmas~\ref{Iab} and \ref{est-dn},
\begin{align*}
    b_n^{[S]}
    &\sim \sum_{2\le k_1,k_2\le 3}
    \frac{(m-1)!A_{k_1}A_{k_2}n^{\lambda_{k_1}+\lambda_{k_2}-2}}
    {\Gamma(\lambda_{k_1})\Gamma(\lambda_{k_2})}\\
    &\qquad\times
    \int_{\substack{x_1+\cdots+x_m=1\\ 0\le x_1,\ldots,x_m\le 1}}
    \left(-1+\sum_{1\le l\le m}x_l^{\lambda_{k_1}-1}\right)
    \left(-1+\sum_{1\le r\le m}x_r^{\lambda_{k_2}-1}\right)
    \mathrm{d}\mathbf{x}\\
    &\sim\sum_{2\le k_1,k_2\le
    3}\frac{A_{k_1}A_{k_2}n^{\lambda_{k_1}+\lambda_{k_2}-2}}
    {\Gamma(\lambda_{k_1})\Gamma(\lambda_{k_2})}
    \Bigg(1-\frac{m!\Gamma(\lambda_{k_1})}
    {\Gamma(\lambda_{k_1}+m-1)}-\frac{m!\Gamma(\lambda_{k_2})}
    {\Gamma(\lambda_{k_2}+m-1)}\\
    &\qquad\qquad+\frac{m!\Gamma(\lambda_{k_1}+\lambda_{k_2}-1)}
    {\Gamma(\lambda_{k_1}+\lambda_{k_2}+m-2)}+\frac{m!(m-1)
    \Gamma(\lambda_{k_1})\Gamma(\lambda_{k_2})}
    {\Gamma(\lambda_{k_1}+\lambda_{k_2}+m-2)}\Bigg).
\end{align*}
Note that
\[
    \frac{m!\Gamma(\lambda_{k_j})}{\Gamma(\lambda_{k_j}+m-1)}
    =1, \qquad (2\le j\le 3).
\]
Applying Proposition \ref{asymp-trans}(ii) term by term then gives
\begin{align*}
    V_n^{[S]}
    &\sim\sum_{2\le k_1,k_2\le 3}
        \frac{A_{k_1}A_{k_2}n^{\lambda_{k_1}+
        \lambda_{k_2}-2}}{\Gamma(\lambda_{k_1})\Gamma(\lambda_{k_2})}
        \left(-1+ \frac{m!(m-1)\Gamma(\lambda_{k_1})
        \Gamma(\lambda_{k_2})}{\Gamma(\lambda_{k_1}
        +\lambda_{k_2}+m-2)-m!
        \Gamma(\lambda_{k_1}+\lambda_{k_2}-1)}\right)\\
    &=: F_1(\beta\log n)n^{2\alpha-2},
\end{align*}
where
\begin{align}\label{F1z}
\begin{split}
    F_1(z)
    &:=2\frac{\vert A_{2}\vert^2}
        {\vert\Gamma(\lambda_{2})\vert^2}\left(-1+\frac{m!(m-1)
        \vert\Gamma(\lambda_{2})\vert^2}{
        \Gamma(2\alpha+m-2)-m!\Gamma(2\alpha-1)}\right)\\
    &\quad\quad+2\Re\left(\frac{A_{2}^2e^{2i z}}
        {\Gamma(\lambda_{2})^2}\left(-1
        +\frac{m!(m-1)\Gamma(\lambda_{2})^2}
        {\Gamma(2\lambda_{2}+m-2)
        -m!\Gamma(2\lambda_{2}-1)}\right)\right).
\end{split}
\end{align}

\paragraph{Asymptotics of $V_n^{[SK]}$.} We now turn to $V_n^{[SK]}$.
If $3\le m\le 13$, then, by Lemma \ref{est-dn},
\[
    b_{n}^{[SK]}
    =O\left(n^{\alpha}\right),
\]
where $\alpha<1$. Consequently, by Proposition \ref{asymp-trans}(i),
\[
    V_n^{[SK]}
    \sim C_R n,
\]
for some constant $C_R$. For the remaining range where $m\ge 14$, we
have $\alpha>1$, and, by Lemma \ref{est-dn} and \eqref{diri-int2},
\begin{align*}
    b_n^{[SK]}
    &\sim\sum_{2\le k\le 3}
        \frac{(m-1)!A_{k}n^{\lambda_k}}
        {\Gamma(\lambda_k)}\int_{\substack{x_1+\cdots+x_m=1\\
        0\le x_1,\ldots,x_m\le 1}}\left(-1+
        \sum_{1\le l\le m}x_l^{\lambda_{k}-1}\right)
        \left(1+2\phi\sum_{1\le r\le m}
        x_r\log x_r\right)\mathrm{d}\mathbf{x}\\
    &\sim \sum_{2\le k\le 3}
        \frac{A_kn^{\lambda_{k}}}{\Gamma(\lambda_k)}
        \Bigg(1-2\phi\frac{m!\Gamma(\lambda_k+1)}
        {\Gamma(\lambda_{k}+m)}
        \bigl\{m\psi(\lambda_k+m)-\psi(\lambda_k+1)
        -(m-1)(1-\gamma)\bigr\}\Bigg).
\end{align*}
Now, we apply Proposition \ref{asymp-trans}(ii) and again after some
straightforward simplifications
\begin{align*}
    V_n^{[SK]}
    &\sim F_2\left(\beta\log n\right)n^{\alpha},
\end{align*}
where
\begin{align}\label{F2z}
\begin{split}
    F_2(z)
    &:=2\phi\Re\Bigg(\frac{(\lambda_2+m-1)
        A_2e^{i z}}{(m-1)\Gamma(\lambda_2)}
        \Bigg(\frac1{2\phi}-\frac{\lambda_2}{\lambda_2+m-1}
        \bigl\{m\psi(\lambda_2+m)-\psi(\lambda_2+1)\\
    &\qquad\qquad-(m-1)(1-\gamma)\bigr\}\Bigg)\Bigg).
\end{split}
\end{align}

\paragraph{Asymptotics of $V_n^{[K]}$.}
In a similar manner, we obtain, by Lemma \ref{est-dn},
\begin{align*}
    b_n^{[\SK]}
    &\sim (m-1)!n^2\int_{\substack{x_1+\cdots+x_m=1\\
        0\le x_1,\ldots,x_m\le 1}}\left(1+2\phi
        \sum_{1\le l\le m}x_l\log x_l\right)^2\text{d}
        \mathbf{x}\\
    &\sim 4\phi^2 n^2\left(H_m^{(2)}-\frac{2}{m+1}
        -\frac{\pi^2(m-1)}{6(m+1)}\right),
\end{align*}
where the last line follows from applying \eqref{diri-int1},
\eqref{diri-int2} and \eqref{diri-int3}. Applying again Proposition
\ref{asymp-trans}(ii) gives
\[
    V_n^{[K]}
    \sim C_\SK n^2,
\]
which completes the proof of Theorem~\ref{cov-snxn}. \qed

\subsection{Correlation between space requirement and NPL}
\label{tnpl}

The calculations in this case are similar to those for
$\rho(S_n,K_n)$, so we only sketch the major steps needed. Briefly,
most asymptotic estimates differ either by a factor of the occupancy
constant $\phi$ or its powers. The only exception is the additional
factor $\log n$ appearing in the covariance $\Cov(S_n,N_n)$ (see
\eqref{rn_coas}).

Let $\nu_n=\E(\SN_n)$. Then
\[
    \nu_n
    =m\sum_{0\le j\le n-m+1}\pi_{n,j}\nu_j+\mu_n-1.
\]
Consequently, by the asymptotic estimate (\ref{exp-mun}) and by
applying Proposition \ref{asymp-trans}(i), we obtain
\begin{equation}\label{exp-etan}
    \nu_n
    =2\phi^2 n\log n+c_2n+o(n),
\end{equation}
where, by Proposition~\ref{asymp-trans},
\begin{align}\label{c2}
	c_2 = \phi c_1 + 2\phi\left(\phi - \frac1{m-1}+
	\sum_{2\le \ell \le m-1}
	\frac{A_\ell}{2-\lambda_\ell}\right),
\end{align}
$c_1$ being given in \eqref{c1} and the $A_\ell$'s defined
in \eqref{Ak}. Indeed, consider the difference $\xi_n := \nu_n - \phi
\kappa_n$, which then satisfies the same recurrence \eqref{und-rec}
but with the toll function
\[
    \eta_n := \mu_n-1-\phi(n-m+1)=
	\phi m -\frac{m}{m-1} + \sum_{2\le \ell <m}
	A_\ell \binom{n+\lambda_\ell-1}{n},
\]
and $\xi_n=0$ for $1\le n\le m-2$. Then by applying Proposition~\ref{asymp-trans}, we obtain
\[
    c_2 - c_1\phi = 2\phi \sum_{j\ge m-1}
	\frac{\eta_j}{(j+1)(j+2)}.
\]
Since $\eta_n =  -\phi (n-m+1)$ for $1\le n\le m-2$ and
$\eta_0=\phi (m -1)-1$, we then derive \eqref{c2} by the relation
\[
	\sum_{j\ge 0} \frac{\binom{\lambda+j-1}{j}}
	{(j+1)(j+2)} = \int_0^1 (1-t)^{-\lambda+1}\dd t
	=\frac1{2-\lambda}\qquad(\Re(\lambda)<2).
\]
% \begin{align*}
% 	\sum_{j\ge 0} \frac{\eta_j}{(j+1)(j+2)}
% 	&= \phi m -\frac{m}{m-1}
% 	+ \sum_{2\le \ell <m}A_\ell
% 	\sum_{j\ge 0} \frac{\binom{\lambda_\ell+j-1}{j}}
% 	{(j+1)(j+2)}\\
% 	&= \phi m -\frac{m}{m-1} +	\sum_{2\le \ell <m}A_\ell
% 	\int_0^1 (1-t)^{-\lambda+1}\dd t\\
% 	&= \phi m -\frac{m}{m-1} +	\sum_{2\le \ell <m}
% 	\frac{A_\ell}{2-\lambda_\ell},
% \end{align*}
% and
% \begin{align*}
% 	-\sum_{0\le j\le m-2}\frac{\eta_j}{(j+1)(j+2)}
% 	= \phi \sum_{0\le j\le m-2}\frac{j-m+1}{(j+1)(j+2)}
% 	+\frac12
% 	= \phi (H_m - m)+\frac12.
% \end{align*}
% Adding these together, we see that
% \begin{align*}
%     c_2 - c_1\phi &= 2\phi
% 	\left(\phi m -\frac{m}{m-1} +	\sum_{2\le \ell <m}
% 	\frac{A_\ell}{2-\lambda_\ell} + \phi (H_m - m)+\frac12\right)\\
% 	&= 2\phi
% 	\left(\phi-\frac{1}{m-1} +	\sum_{2\le \ell <m}
% 	\frac{A_\ell}{2-\lambda_\ell} \right).
% \end{align*}
In particular, $c_2-\phi c_1$ equals
\[
    \tfrac{12}{125}, \tfrac{222}{2197}, \tfrac{44670}{456533},
	\tfrac{710}{7569}, \tfrac{8990170}{99806103},
	\tfrac{86959460}{1001561769},
	\tfrac{8225243460}{97908438529},
	\tfrac{9368632980}{114862129381},
\]
for $m=3,\dots,10$, and
\begin{tiny}
\begin{align*}
	&\frac{13941168359580}{175531341607271},
	\frac{15364018080180}{198165483844901},
	\frac{36778736979244260}{484907780151231137},
	\frac{39706104830251860}{534148059351752117},
	\frac{42542306175669300}{583013664848115773}, \\
	&\frac{362341148683714200}{5051607560589134719},
	\frac{60809828396490973800}{861420713064800471777},
	\frac{220781849887636437400}
	{3174476111482140491583},
	\frac{1589879045909940738152200}
	{23180880112213178399314917},\\ &
	\frac{66535629228892650939112}
	{982905224931956375768865},
	\frac{69399644946307963559272}
	{1037954891250806970920625},
	\frac{72191400913204902200872}
	{1092384284013327674677545}, \\ &
	\frac{911488027263952226045421464}
	{13945777153309079949132939375},
	\frac{943834826916499599456679304}
	{14593082411910111966602252205},
	\frac{3048229719576792424490262245800}
	{47603282606571951420821994029889}
	,\\ &
	\frac{3144754504512378111611222765800}
	{49580602253255626178697360169689},%\\ &
	\frac{787117453959995151898324789769400}
	{12523181563980976087610969389067627},
	\frac{809570585901011449194661971389400}
	{12992983079952314295925927936613927}, \\ &
	\frac{20280854972612671613961769087339836600}
	{328217277361176269245342166728792498003},
	\frac{20806237502125190663861808383733444600}
	{339424705221771320114642916145949390923}
\end{align*}
\end{tiny}
for $m=11,\dots,30$.

Let $\bar{\SN}_n=\SN_n-\nu_{n}$. Then the moment-generating function
$\bar{P}_n(u,v) :=\E\bigl(e^{\bar{S}_n u+\bar{\SN}_nv}\bigr)$
satisfies for $n\ge m-1$
\[
    \bar{P}_n(u,v)
    =\frac{1}{\binom{n}{m-1}}
        \sum_{\mathbf{j}}P_{j_1}(u+v,v)\cdots
        P_{j_m}(u+v,v)e^{\Delta_{\mathbf{j}}u
        +\delta_{\mathbf{j}}v},
\]
with the initial conditions $\bar{P}_n(u,v)=1$ for $0\le n\le m-2$
and
\[
    \delta_{\mathbf{j}}
    :=-\nu_n+ \sum_{1\le l\le m}
        \left(\nu_{j_{l}}+\mu_{j_l}\right).
\]
Now define
\[
    V^{[SN]}_n
    :=\Cov(S_n,\SN_n)
        \qquad\text{and}\qquad
    V^{[N]}_n
    :=\V(\SN_n).
\]
Then
\begin{align*}
    V^{[X]}_n
    &=m\sum_{0\le l\le n-m+1}\pi_{n,j}V^{[X]}_j+b_{n}^{[X]},
    \qquad(X\in\{SN,N\}),
\end{align*}
where
\begin{align*}
    b_{n}^{[SN]}
    &=\frac{1}{\binom{n}{m-1}}
        \sum_{\mathbf{j}}\left(V_j^{[S]}
        +\Delta_{\mathbf{j}}\delta_{\mathbf{j}}\right)\\
    &=V_n^{[S]}+\frac{1}{\binom{n}{m-1}}\sum_{\mathbf{j}}
        \left(\Delta_{\mathbf{j}}\delta_{\mathbf{j}}
        -\Delta_{\mathbf{j}}^2\right)\\
    b_{n}^{[N]}
    &=\frac{1}{\binom{n}{m-1}}\sum_{\mathbf{j}}
        \left(V_j^{[S]}+2V^{[SN]}_j+\delta_{\mathbf{j}}^2\right)\\
    &=V_n^{[S]}+2V^{[SN]}_n+\frac{1}{\binom{n}{m-1}}
        \sum_{\mathbf{j}}\left(\delta_{\mathbf{j}}^2
        -2\Delta_{\mathbf{j}}\delta_{\mathbf{j}}
        +\Delta_{\mathbf{j}}^2\right).
\end{align*}

As in the case of KPL, the following uniform estimate is crucial in
our analysis.
\begin{lem}\label{est-d}
Uniformly in $\mathbf{j}$,
\[
    \delta_{\mathbf{j}}
    =\phi n\left(1+2\phi
        \sum_{1\le l\le m}\frac{j_l}{n}\log\frac{j_l}{n}\right)
        +o(n).
\]
\end{lem}
\pf By the definition of $\delta_{\mathbf{j}}$ and the estimates
(\ref{exp-mun}) and (\ref{exp-etan}). \qed

Note that the expansion differs from that for $\nabla_{\mathbf{j}}$
in Lemma \ref{est-dn} by an additional factor $\phi$.

If $3\le m\le 13$, then, by Lemmas \ref{est-dn} and \ref{est-d},
\[
    b_{n}^{[SN]}
    = C_S n+O\left(n^{1-\ve}\right),
\]
for a sufficiently small $\ve>0$. Thus, by Proposition
\ref{asymp-trans} (i),
\[
    V^{[SN]}_n
    \sim \frac{C_Sn\log n}{H_m-1}.
\]
Assume now $m\ge 14$. Then, again from Lemma \ref{est-dn} and Lemma
\ref{est-d} together with the known asymptotics of $V_n^{[S]}$, we
see that
\[
    b_n^{[SN]}
    \sim \frac{1}{\binom{n}{m-1}}
        \sum_{\mathbf{j}}\Delta_{\mathbf{j}}\delta_{\mathbf{j}}
    \sim\frac{\phi}{\binom{n}{m-1}}
        \sum_{\mathbf{j}}\Delta_{\mathbf{j}}\nabla_{\mathbf{j}}
    \sim \phi b_n^{[SK]}.
\]
Thus we deduce, as in the proof for $V_n^{[SK]}$,
\[
    V_n^{[SN]} \sim \phi V_n^{[SK]} \sim
    \phi F_2(\beta\log n) n^\alpha.
\]
Similarly, we have
\[
    b_n^{[\SN]}
    \sim\frac{1}{\binom{n}{m-1}}
        \sum_{\mathbf{j}}\delta_{\mathbf{j}}^2
    \sim\frac{\phi^2}{\binom{n}{m-1}}
        \sum_{\mathbf{j}}\nabla_{\mathbf{j}}^2
    \sim \phi^2 b_n^{[\SK]}.
\]
Consequently,
\[
    V_n^{[N]} \sim \phi^2 V_n^{[K]}\sim
    \phi^2 C_K n^2.
\]
This completes the proof of Theorem~\ref{rn_coas}. \qed

\section{Bivariate distributional asymptotics for space requirement
and path lengths}

\label{sec_bivariate}

In this section, we identify the asymptotic joint distributional
behaviors of the pairs $(N_n,S_n)$ and $(K_n,S_n)$. Although the
sequences $(N_n)$ and $(K_n)$ converge after normalization for all
$m\ge 3$ with limit distributions depending on $m$, we split the
analysis into two cases depending on $3\le m\le 26$ or $m>26$ due to
the phase change in the limit behavior of $S_n$. We discuss the pair
$(N_n,S_n)$ in detail in Sections~\ref{secns<=26} and \ref{secns>26}.
(the corresponding analysis for $(K_n,S_n)$ is similar and we will
not give details). Moreover, in Section \ref{univariate}, we will
show that the univariate limit random variables of the normalized
sequences $(N_n)$ and $(K_n)$ do have the same distribution. We
introduce the following notation
\begin{align}\label{exp_mu}
    \mu(n)
    :=\mu_n = \E[S_n]=\phi (n+1)
        + \Re(\theta n^{\lambda_2-1})
        + o(1\vee n^{\alpha-1}) ,
\end{align}
where $\theta := 2A_2/\Gamma(\lambda_2)$; see \eqref{exp-mun}.
Similarly, write $\kappa(n) = \kappa_n=\E(K_n)$ and $\nu(n) = \nu_n
=\E(N_n)$.

\subsection{Node path length and space requirement. I.
$\mathbf{m\ge 27}$}
\label{secns<=26}

We give in this section the precise formulation of the periodic case
$m\ge27$ of Theorem \ref{thm12}.

\paragraph{Normalization.} We first normalize the vector
$Q_n=(N_n,S_n)$ as follows. Let $Y_0:=0$ and
\begin{align*}%\label{norma0}
    Y_n
    :=\left(\frac{N_n-\E[N_n]}{n},\frac{S_n-\phi n}
        {n^{\alpha-1}}\right), \qquad (n\ge 1).
\end{align*}
Then the recurrence (\ref{rec1}) implies for $n\ge m-1$
\begin{align} \label{mod_rec}
    \left(Y_n\right)^{\mathrm{t}}
    \stackrel{d}{=}\sum_{1\le r\le m}
        A^{(n)}_r \left(Y^{(r)}_{I^{(n)}_r}
    \right)^{\mathrm{t}}+ b^{(n)},
\end{align}
where
\begin{align*}
    A^{(n)}_r
    :=\left[
    \begin{array}{cc}\displaystyle
        \frac{I^{(n)}_r}{n} & \displaystyle
        \frac{\bigl(I^{(n)}_r\bigr)^{\alpha-1}}{n}\\
        0&\displaystyle
		\left(\frac{I^{(n)}_r}{n}\right)^{\alpha-1}
    \end{array}\right], \quad
    b^{(n)}:= \left(\!\!
    \begin{array}{c}\displaystyle
        \frac{1}{n}\left(\sum_{1\le r\le m}
            \left(\nu\bigl(I^{(n)}_r\bigr)
            +\phi I^{(n)}_r\right) -\nu(n)\right)\\
        \displaystyle
        -\phi\frac{m-1}{n^{\alpha-1}}
    \end{array}\!\!\right),
\end{align*}
with assumptions on independence and on identical distributions
as in Section~\ref{sec_pre}. The expansion (\ref{exp-etan}) implies
\begin{align*}
    \frac{1}{n}\left(\sum_{1\le r\le m}
        \left(\nu\bigl(I^{(n)}_r\bigr)
        +\phi I^{(n)}_r\right) -\nu(n)\right)
    = \phi + 2\phi^2\sum_{1\le r\le m}
        \frac{I^{(n)}_r}{n} \log \frac{I^{(n)}_r}{n}
        + o(1).
\end{align*}
Moreover, by \eqref{lim_ij}, we obtain the $L_2$-convergence
\begin{align}\label{conv_in}
    \frac{I^{(n)}}{n}
    \stackrel{L_2}{\longrightarrow} (V_1,\ldots,V_m)=:V.
\end{align}
This implies the $L_2$-convergences
\begin{align}\label{def_b1}
    \frac{1}{n}\left(\sum_{1\le r\le m}
        \left(\nu\bigl(I^{(n)}_r\bigr)
        +\phi I^{(n)}_r\right) -\nu(n)\right)
    \to \phi + 2\phi^2\sum_{1\le r\le m}  V_r \log
        V_r=:b_N,
\end{align}
and
\begin{align}\label{conv_bn}
    b^{(n)}
    \to \left(\!\!
    \begin{array}{c}
        b_N \\ 0
    \end{array}\!\!\right),\qquad
    A^{(n)}_r
    \to \left[
    \begin{array}{cc}
        V_r & 0\\ 0& V_r^{\alpha-1}
    \end{array}\right] .
\end{align}

For our limit result for $m\ge 27$, we first define a distribution
which governs the asymptotics.

\paragraph{The limiting map.} To describe the asymptotic behavior of
$Q_n$, we use the following probability distribution on the space $\R
\times \C$. Let $\mathcal{M}^{\R \times \C}$ denote the space of all
distributions $\mathcal{L}(Z,W)$ on $\R \times \C$ and
$\mathcal{M}^{\R \times \C}_2$ the subspace of distributions with
finite second moment, i.e., $\|(Z,W)\|_2:= (\E[Z^2]+\E[|W|^2])^{1/2}
<\infty$. For $\gamma=(\gamma_1,\gamma_2) \in \R \times \C$, let
\begin{align*}
    \mathcal{M}^{\R \times \C}_2(\gamma)
    :=\Big\{\mathcal{L}(Z,W)\in
        \mathcal{M}^{\R \times \C}_2\,\Big|\, \E[Z]=\gamma_1,
    \E[W]
    =\gamma_2\Big\}.
\end{align*}
We define the following map $T_N$ on $\mathcal{M}^{\R \times \C}_2$:
\begin{align}
    T_N: \mathcal{M}^{\R \times \C}
    &\to\mathcal{M}^{\R \times \C}  \nonumber\\
    \mathcal{L}(Z,W)
    &\mapsto \mathcal{L}\left( \sum_{1\le r\le m}
    \left[
    \begin{array}{cc}
        V_r & 0\\ 0& V_r^{\lambda_2-1}
    \end{array}\right]\left(\!\!
    \begin{array}{c}
        Z^{(r)} \\ W^{(r)}
    \end{array}\!\!\right) + \left(\!\!
    \begin{array}{c}
        b_N \\ 0
    \end{array}\!\!\right) \right),
    \label{def_map_t}
\end{align}
where $(Z^{(1)},W^{(1)}),\ldots, (Z^{(m)},W^{(m)})$, $V$ are
independent, $(Z^{(r)},W^{(r)})$ is distributed as $(Z,W)$ for all
$r=1,\ldots,m$ and $b_N$ is defined in (\ref{def_b1}). The
$\|\,\cdot\,\|_2$-norm induces the minimal $L_2$-metric $\ell_2$ by
\begin{align*}
    \ell_2(\mu,\nu)
    := \inf\{ \| X-Y\|_2\,:\,
    \mathcal{L}(X)
    =\mu, \mathcal{L}(Y)=\nu\}, \qquad (\mu,\nu \in
        \mathcal{M}^{\R \times \C}_2).
\end{align*}
Given random variables $X,Y$, write for simplicity
$\ell_2(X,Y)=\ell_2(\mathcal{L}(X),\mathcal{L}(Y))$. For any
distributions $\mu,\nu \in \mathcal{M}^{\R \times \C}_2$, there exist
optimal $\ell_2$-couplings, i.e.~random vectors $\Upsilon_1,
\Upsilon_2$ in $\R \times \C$ with $\ell_2(\mu,\nu) = \|
\Upsilon_1-\Upsilon_2\|_2$.

\begin{lem} \label{fplem1} Assume $m\ge 27$. For any $\gamma \in \R
\times \C$, the restriction of the map $T_N$ defined in
(\ref{def_map_t}) to $\mathcal{M}^{\R \times \C}_2(\gamma)$ is a
(strict) contraction with respect to $\ell_2$, and has a unique fixed
point in $\mathcal{M}^{\R \times \C}_2(\gamma)$.
\end{lem}
\begin{proof} Let $\gamma \in \R \times \C$ be arbitrary. For $\mu
\in \mathcal{M}^{\R \times \C}_2(\gamma)$, let $\Upsilon$ be a random
variable with distribution $T(\mu)$. First, note that
$\|\Upsilon\|_2<\infty$ by independence and $\|b_N\|_2 <\infty$ (we
even have $\|b_N\|_\infty <\infty$). To see that
$\E[\Upsilon]=\gamma$, note that $\E[b_N]=0$ and $\sum_{1\le r\le m}
V_r=1$ almost surely. Hence, we only need to show that
$\E[V_1^{\lambda_2-1}]=1/m$. Since $V_1$ has density $x\mapsto
(m-1)(1-x)^{m-2}$ for $x\in[0,1]$, we see that
\begin{align*}
    \E\left[V_1^{\lambda_2-1}\right]
    &=\int_0^1(m-1)(1-x)^{m-2}x^{\lambda_2-1}\,\text{d}x
    =(m-1)\frac{\Gamma(m-1)\Gamma(\lambda_2)}
        {\Gamma(m+\lambda_2-1)}=\frac{1}{m},
\end{align*}
because $\Gamma(m+\lambda_2-1)/\Gamma(\lambda_2)=m!$. This implies
that $\E[\Upsilon]=\gamma$, and thus $T(\mu)\in \mathcal{M}^{\R
\times \C}_2(\gamma)$. This in turn implies that the restriction of
$T$ to $\mathcal{M}^{\R \times \C}_2(\gamma)$ maps into
$\mathcal{M}^{\R \times \C}_2(\gamma)$.

That the restriction of $T$ to $\mathcal{M}^{\R \times \C}_2(\gamma)$
is a contraction with respect to $\ell_2$ follows from a standard
calculation, e.g., with a slight modification as in \cite[Lemma
3.1]{neininger01}.
\end{proof}

\paragraph{Proof of Theorem~\ref{thm12}: NPL.} Denote by
$\mathcal{L}(X,\Lambda)$ the unique fixed point of the restriction of
$T_N$ to $\mathcal{M}^{\R \times \C}_2((0,\theta))$, with $\theta$
defined in (\ref{exp_mu}). By Lemma \ref{fplem1}, the distribution
$\mathcal{L}(X,\Lambda)$ as in the statement of the Theorem is
well-defined. The fixed point property of $(X,\Lambda)$ implies that
\begin{align}\label{mod_fix}
    \left(\!\! \begin{array}{c}
        X \\ \Re(n^{i\beta} \Lambda)
    \end{array}\!\!\right)
    \stackrel{d}{=} \left(\!\!
    \begin{array}{c} \displaystyle
        \sum_{1\le r\le m} V_r X^{(r)} +b_N\\
        \displaystyle
        \sum_{1\le r\le m} \Re(n^{i\beta}
            V_r^{\lambda_2-1}\Lambda^{(r)})
    \end{array}\!\!\right),
\end{align}
where $(V_1,\ldots,V_m)$, $(X^{(1)},\Lambda^{(1)}),\ldots,
(X^{(m)},\Lambda^{(m)})$ are independent, and
$(X^{(r)},\Lambda^{(r)})$ are identically distributed as
$(X,\Lambda)$.

Define now three matrices
\begin{align*}
    \widetilde{A}^{(n)}_r&:= \left[
    \begin{array}{cc} \displaystyle
        \frac{I^{(n)}_r}{n} &0 \\
        0& \displaystyle
        \left(\frac{I^{(n)}_r}{n}\right)^{\alpha-1}
    \end{array}\right], \;
    B^{(n)}_r:= \left[
    \begin{array}{cc}
        V_r &0\\
        0& n^{i\beta} V_r^{\lambda_2-1}
    \end{array}\right],\;
    C^{(n)}_r&:= \left[
    \begin{array}{cc} \displaystyle
        \frac{I^{(n)}_r}{n} &0\\
        0& \displaystyle
        \frac{\left(I^{(n)}_r\right)^{\lambda_2-1}}{n^{\alpha-1}}
    \end{array}\right],
\end{align*}
and write
\begin{align*}
    \Delta(n)
    :=\ell_2(Y_n,(X,\Re(n^{i\beta}\Lambda))).
\end{align*}
To bound $\Delta(n)$, we use the following coupling between the
$Y^{(r)}_j$'s appearing in the recurrence (\ref{mod_rec}) and the
quantities appearing on the right-hand side of (\ref{mod_fix}). Note
that for any pair of distributions on $\R^2$, there always exists an
optimal $\ell_2$-coupling. We first fix the random vectors
$(X^{(1)},\Lambda^{(1)}),\ldots, (X^{(m)},\Lambda^{(m)})$. Then, for
each $j\ge 1$ and $r=1,\ldots,m$, we choose $Y^{(r)}_j$ as an optimal
$\ell_2$-coupling to $(X^{(r)},\Re(j^{i\beta}\Lambda^{(r)}))$ on
$\R^2$. This can be done such that the sequences
$$\left(Y^{(1)}_j,(X^{(1)},\Re(j^{i\beta}\Lambda^{(1)}))\right)_{j\ge
1},\ldots, \left(Y^{(m)}_j,(X^{(m)},\Re(j^{i\beta} \Lambda^{(m)}
))\right)_{j\ge 1}$$ are independent and independent of
$(I^{(n)},V_1,\ldots,V_m)$. Note that these couplings and
independence assumptions do not violate equations (\ref{mod_rec}) and
(\ref{mod_fix}). Hence, we obtain
\begin{align*}
    \Delta(n)
    \le \left\|\sum_{1\le r\le m}
        A^{(n)}_r\left(Y^{(r)}_{I^{(n)}_r}\right)^t
        +b^{(n)} -\Re \left( \sum_{1\le r\le m} B^{(n)}_r
        \left(\!\!
    \begin{array}{c}
        X^{(r)} \\  \Lambda^{(r)}
    \end{array}\!\!\right) + \left(\!\!
    \begin{array}{c}
        b\\ 0
    \end{array}\!\!\right) \right) \right\|_2.
\end{align*}
Using the triangle inequality and writing the components as
$Y_n=(Y_{n,1},Y_{n,2})$, we obtain
\begin{align*}
    \Delta(n)
    \le &  \left\| \sum_{1\le r\le m}
        \left(
		\widetilde{A}^{(n)}_r\left(Y^{(r)}_{I^{(n)}_r}\right)^t-
        \Re \left(C^{(n)}_r  \left(\!\! \begin{array}{c}
        X^{(r)} \\  \Lambda^{(r)} \end{array}\!\!\right)
        \right) \right)  \right\|_2 \\
    & +  \left\| \sum_{1\le r\le m} \Re \left(C^{(n)}_r \left(\!\!
    \begin{array}{c}
        X^{(r)} \\  \Lambda^{(r)}
    \end{array}\!\!\right)\right)
        - \Re\left(B^{(n)}_r  \left(\!\!
    \begin{array}{c}
        X^{(r)} \\  \Lambda^{(r)}
    \end{array}\!\!\right) \right) \right\|_2\\
    &+ \sum_{1\le r\le m} \left\|
        \frac{(I^{(n)}_r)^{\alpha-1}}{n}
        Y^{(r)}_{I^{(n)}_r, 2} \right\|_2
        + \left\| b^{(n)} - \left(\!\!
    \begin{array}{c}
        b\\ 0
    \end{array}\!\!\right) \right\|_2.
\end{align*}
The second and the fourth summand on the right-hand side tend to zero
as $n\to \infty$ by (\ref{conv_in}) and (\ref{conv_bn}). For the
third summand, note that the asymptotic behavior of the normalized
size $Y_{n,2}$ of $m$-ary search trees is covered by Theorem 1,
eq.~(2) in \cite{chern01}. In particular, from that theorem we obtain
$\sup_{n\ge 1} \|Y_{n,2}\|_2<\infty$. Taking into account the
prefactor $(I^{(n)}_r)^{\alpha-1}/n$ and conditioning on $I^{(n)}_r$,
we find that the third summand also tends to zero.

To bound the first summand in the latter display, we write, for
$r=1,\ldots,m$ and $n\ge m-1$,
\begin{align*}
    W^{(n)}_r
    := \widetilde{A}^{(n)}_r\left(Y^{(r)}_{I^{(n)}_r}\right)^t-
        \Re \left(C^{(n)}_r  \left(\!\!
    \begin{array}{c}
        X^{(r)} \\  \Lambda^{(r)}
    \end{array}\!\!\right)  \right)
\end{align*}
and denote the components of $W^{(n)}_r$ by $W^{(n)}_r
=(W^{(n)}_{r,1}, W^{(n)}_{r,2})$. For $r=1,\ldots,m$, we have
\begin{align}
    \E \left\| \sum_{1\le r\le m} W^{(n)}_r \right\|^2
    = \E \left[  \sum_{1\le r\le m} \left\{(W^{(n)}_{r,1})^2
        + (W^{(n)}_{r,2})^2\right\} +
        \sum_{r\neq s} \left\{W^{(n)}_{r,1}W^{(n)}_{s,1} +
        W^{(n)}_{r,2}W^{(n)}_{s,2}\right\}\right].
    \label{dec_wn}
\end{align}
We bound the three types of terms individually. First, for
the dominant term
\begin{align*}
    \lefteqn{ \E \left[ (W^{(n)}_{r,1})^2
    + (W^{(n)}_{r,2})^2\right]}\\
    &= \E \left[ \left(\frac{I^{(n)}_r}{n}\right)^2 \left(
    Y^{(r)}_{I^{(n)}_r,1} -X^{(r)}\right)^2   +
    \left(\frac{I^{(n)}_r}{n}\right)^{2(\alpha-1)} \left(
    Y^{(r)}_{I^{(n)}_r,2} - \Re\left( (I^{(n)}_r)^{i\beta}
    \Lambda^{(r)}\right)\right)^2\right]
\end{align*}
\newpage
\begin{align*}
    &\le \E \left[
    \left(\frac{I^{(n)}_r}{n}\right)^{2(\alpha-1)}
    \left(\left( Y^{(r)}_{I^{(n)}_r,1} -X^{(r)}\right)^2 +
    \left(Y^{(r)}_{I^{(n)}_r,2} - \Re\left( (I^{(n)}_r)^{i\beta}
    \Lambda^{(r)}\right)\right)^2  \right) \right]
\end{align*}
where we used the inequality $(I^{(n)}_r/n)^2 \le
(I^{(n)}_r/n)^{2(\alpha-1)}$. Conditioning on $I^{(n)}_r$ and using
that $Y^{(r)}_j$ and $(X^{(r)},\Re(j^{i\beta}\Lambda^{(r)}))$ are
optimal couplings, we obtain
\begin{align*}
    \E \left[ (W^{(n)}_{r,1})^2 + (W^{(n)}_{r,2})^2\right]
    \le \E\left[\left(\frac{I^{(n)}_r}{n}\right)^{2(\alpha-1)}
    \Delta^2(I^{(n)}_r)\right].
\end{align*}
For the cross-product terms in (\ref{dec_wn}), assume $1\le r,s\le m$
with $r\ne s$. Note that, by independence, we have $\E[W^{(n)}_{r,1}
W^{(n)}_{s,1}]=0$ conditioning on $I^{(n)}_r$ and $I^{(n)}_s$. From
the expansion (\ref{exp_mu}), we obtain
\begin{align*}
    \E \left[Y_n\right] =  \left(\!\! \begin{array}{c} 0 \\
    \Re(\theta n^{i\beta}) + R(n) \end{array}\!\!\right),
\end{align*}
with a remainder $R(n)=o(1)$. By independence and
$\E[\Lambda]=\theta$, we obtain $\E[W^{(n)}_{r,2}]=
\E[(I^{(n)}_r/n)^{\alpha-1} R(I^{(n)}_r)]$, and
\begin{align*}
    \E[W^{(n)}_{r,2}W^{(n)}_{s,2}]
    =\E\left[\left(\frac{I^{(n)}_r}{n}\cdot
    \frac{I^{(n)}_s}{n}\right)^{\alpha-1}
    R(I^{(n)}_r)R(I^{(n)}_s) \right]
\end{align*}
which tends to $0$ by the dominated convergence theorem as
$R(I^{(n)}_r), R(I^{(n)}_s)\to 0$ in probability.

Hence, collecting all estimates, we obtain
\begin{align}\label{est_delta}
    \Delta(n) \le \left(\E\left[\sum_{1\le r\le m}
    \left(\frac{I^{(n)}_r}{n}\right)^{2(\alpha-1)}
    \Delta^2(I^{(n)}_r) \right] +o(1)  \right)^{1/2} + o(1).
\end{align}
Now $\Delta(n)\to 0$ follows from a standard argument since we have
\begin{align*}
    \lim_{n\to\infty} \sum_{1\le r\le m}
    \E \left[\left(\frac{I^{(n)}_r}{n}\right)^{2(\alpha-1)}
    \right] = \sum_{1\le r\le m} \E\left[V_r^{2(\alpha-1)}\right]
    = m^2B(m,2\alpha-1)<1;
\end{align*}
(cf.\ the proof of Theorem 4.1 in \cite{neininger01}). This proves
Theorem~\ref{thm12} for NPL.

\subsection{Node path length and space requirement. II.
$\mathbf{3\le m\le 26}$}
\label{secns>26}

We begin with the recurrence (\ref{rec1}), and recall that, for $3\le
m\le 26$,
\begin{align*}
    \V(S_n)\sim C_S n,\quad \V(N_n)\sim C_N n^2 \quad
    \text{with} \quad C_N=\phi^2C_K;
\end{align*}
see \eqref{exp-nun} and \eqref{exp-etan}. There exists an $n_1\ge 1$,
such that for all $n\ge n_1$, the matrix $\Cov(Q_n)$ is positive
definite. We normalize it by $\widetilde{Q}_n:=Q_n$ for $0\le n<n_1$
and by
\begin{align*}
    \left(\widetilde{Q}_n\right)^{\mathrm{t}}
    :=\left[\begin{array}{cc}
    (\sqrt{C_N} n)^{-1} & 0\\ 0& (C_Sn)^{-1/2} \end{array}\right]
    \left(Q_n-\E[Q_n]\right)^{\mathrm{t}}, \qquad (n\ge n_1).
\end{align*}
Then, by (\ref{rec1}), $\widetilde{Q}_n$ satisfies the recurrence
\begin{align*}
    \left(\widetilde{Q}_n\right)^{\mathrm{t}}
    \stackrel{d}{=} \sum_{1\le r\le m} D^{(n)}_r
    \left(\widetilde{Q}^{(r)}_{I^{(n)}_r}\right)^{\mathrm{t}}
    + \widetilde{b}_n, \qquad (n\ge m-1),
\end{align*}
where (denoting by $F_{n,r}$ the event $F_{n,r}:=\{I^{(n)}_r\ge
n_1\}$ and $F_{n,r}^c$ its complement)
\begin{align} \label{coeml26}
    D^{(n)}_r
    &= \left[
    \begin{array}{cc} \displaystyle
	    \left(\frac{I^{(n)}_r}{n}\right)
        \mathbf{1}_{F_{n,r}} +
        \frac{\mathbf{1}_{F_{n,r}^c}}{\sqrt{C_N}\, n}
           & \displaystyle
           \frac{\sqrt{C_S I^{(n)}_r}}{\sqrt{C_N}\, n}
           \, \mathbf{1}_{F_{n,r}}
           + \frac{\mathbf{1}_{F_{n,r}^c}}{\sqrt{C_N} n}     \\
        0& \displaystyle
        \frac{\sqrt{I^{(n)}_r}}{\sqrt{n}}\,\mathbf{1}_{F_{n,r}}
            + \frac{\mathbf{1}_{F_{n,r}^c}}{\sigma_Y
        \sqrt{n}}
    \end{array}\right], \nonumber\\
    \widetilde{b}_n
    &=  \left(\!\!
    \begin{array}{c} \displaystyle
        \frac1{\sqrt{C_N}\, n}\left(\sum_{1\le r\le m}
            (\nu (I^{(n)}_r)+\mu(I^{(n)}_r)) -\nu(n)\right)\\
        \displaystyle
        \frac1{C_sn}\left(1-\mu(n)+\sum_{1\le r\le m}
            \nu (I^{(n)}_r)\right)
    \end{array}\!\!\right),
\end{align}
with assumptions on independence and identical distributions as in
(\ref{rec1}). Note that the asymptotic expressions for the variances
and covariance between $N_n$ and $S_n$ imply that
\begin{align*}
    \Cov(\widetilde{Q}_n)
    = \mathrm{Id_2} + o(1) ,
\end{align*}
where $\mathrm{Id_2}$ denotes the $2\times 2$ identity matrix and the
$o(1)$-term means that all four components of $\Cov(\widetilde{Q}_n)$
converge to the corresponding components of $\mathrm{Id_2}$, each
$o(1)$ in the four components being different in general. In
particular, $\Cov(\widetilde{Q}_n)$ is a symmetric, positive definite
matrix for all $n\ge n_1$. Let $R_n:=\mathrm{Id}_2$ for $0\le n <
n_1$ and $R_n :=(\Cov(\widetilde{Q}_n))^{1/2}$ for $n\ge n_1$. Note
that, by continuity, we have
\begin{align}\label{conRn}
    R_n
    = \mathrm{Id_2} + o(1), \quad R_n^{-1}=\mathrm{Id_2} + o(1).
\end{align}
Now normalize $\widetilde{Q}_n$ by $Y_n := R_n^{-1}\widetilde{Q}_n$,
for $n\ge1$, so that $\Cov(Y_n)=\mathrm{Id_2}$ for $n\ge n_1$, and
\begin{align}\label{yn_norm}
    \left(Y_n\right)^{\mathrm{t}}
    \stackrel{d}{=} \sum_{1\le r\le m} F^{(n)}_r
    \left(Y^{(n)}_{I^{(n)}_r}\right)^{\mathrm{t}}
    + b^{(n)}, \qquad (n\ge n_1),
\end{align}
where $F^{(n)}_r
    = R_n^{-1} D^{(n)}_r R_{I^{(n)}_r}$ and $
    b^{(n)}
    = R_n^{-1} \widetilde{b}_n,
$
with assumptions on independence and identical distributions as in
(\ref{rec1}). From (\ref{coeml26}), (\ref{conRn}) and
(\ref{conv_in}), we then obtain the convergences
\begin{align}\label{coeffyn}
    F^{(n)}_r
    \to \left[
    \begin{array}{cc}
        V_r & 0\\ 0&  V_r^{1/2}
    \end{array}\right]=:F^*_r, \quad b^{(n)}
    \to  \left(\!\!
    \begin{array}{c}
        C_N^{-1/2} b_N\\ 0
    \end{array}\!\!\right)=:b_N^*,
\end{align}
which hold in $L_p$ for any $1\le p<\infty$ (we will need $p=3$
below).

\paragraph{The limiting map.} To describe the asymptotic behavior of
$Q_n$, we use the following probability distribution on the space
$\R^2$. In accordance with the notation in \cite{neininger04}, we
denote by $\mathcal{M}^{2}$ the space of all probability
distributions on $\R^2$, by $\mathcal{M}^{2}_3$ the subspace of all
$\mathcal{L}(Z)\in \mathcal{M}^{2}$ with $\|Z\|_3<\infty$, and
furthermore
\begin{align*}
    \mathcal{M}^{2}_3(0,\mathrm{Id_2})
    :=\Big\{\mathcal{L}(Z)\in \mathcal{M}^{2}_3\,\Big|\,
    \E[Z]=0, \Cov(Z)= \mathrm{Id_2}\Big\}.
\end{align*}
Define the map $T_N'$ on $\mathcal{M}^{2}$:
\begin{align}
    T'_N: \mathcal{M}^{2}
    &\to \mathcal{M}^{2}, \label{def_map_t'}\\
    \mathcal{L}(Z)
    &\mapsto \mathcal{L}\left( \sum_{1\le r\le m}
        F^*_r Z^{(r)}+ b_N^* \right),\nonumber
\end{align}
where $Z^{(1)},\ldots, Z^{(m)}$, $(F^*_1,\ldots,F^*_m,b_N^*)$ are
independent and $Z^{(r)}$ is distributed as $Z$ for all
$r=1,\ldots,m$. Here $F^*_r$ and $b_N^*$ are defined in
(\ref{coeffyn}).

\begin{lem}\label{fpm23} The restriction of $T_N'$ in
(\ref{def_map_t'}) to $\mathcal{M}^{2}_3(0,\mathrm{Id_2})$ has a
unique fixed point $\mathcal{L}(X',\Lambda')$ which is a product
measure, i.e., its components $X'$ and $\Lambda'$ are independent.
\end{lem}

\begin{proof}
We check first that the restriction of $T_N'$ to
$\mathcal{M}^{2}_3(0,\mathrm{Id_2})$ maps into
$\mathcal{M}^{2}_3(0,\mathrm{Id_2})$:
\begin{itemize}
	
\item For any $\mu \in \mathcal{M}^{2}_3(0,\mathrm{Id_2})$, we see, by
independence and $\|b_N\|_3<\infty$, that $T_N'(\mu)\in
\mathcal{M}^{2}_3$.

\item For the mean of $T_N'(\mu)$, we have, from $\E[b_N]=0$, that
$T_N'(\mu)$ is centered.

\item For the covariance of $T_N'(\mu)$, we obtain (see also
\cite[Lemma 3.2]{neininger04}) the matrix
\begin{align} \label{cov-id}
    \E \left[
    \begin{array}{cc}
        b_N^2/C_N & 0\\
        0& 0 \end{array}
    \right] + m \E \left[
    \begin{array}{cc}
        V_1^2 & 0\\
        0& V_1 \end{array}
    \right]
    = \mathrm{Id_2}.
\end{align}
\end{itemize}
Thus $T_N'(\mu)\in\mathcal{M}^{2}_3(0,\mathrm{Id_2})$. By Lemma 3.3
in \cite{neininger04}, the existence of a unique fixed point
$\mathcal{L}(X',\Lambda')$ follows from the inequality
\begin{align*}
    m \E \left\|F_1^*\right\|_{\mathrm{op}}^3 = m\E[V_1^{3/2}]<1.
\end{align*}
Alternatively, Theorem 5.1 in \cite{drmota08} (or Lemma 3.1 in
\cite{neininger04} as well) implies the existence of a unique fixed
point $\mathcal{L}(X',\Lambda')$ in
$\mathcal{M}^{2}_3(0,\mathrm{Id_2})$.

To show that $\mathcal{L}(X',\Lambda')$ is a product measure we
recall that the existence of the unique fixed point that we just
obtained is based on the fact that the restriction of $T'_N$ to
$\mathcal{M}_3^2(0,\mathrm{Id}_2)$ is a contraction with respect to a
complete metric on $\mathcal{M}_3^2(0,\mathrm{Id}_2)$. We do not
introduce this metric, the Zolotarev metric $\zeta_3$, here, since we
do not require the special description of $\zeta_3$. For more
information on $\zeta_3$, in particular the completeness of the
metric space $(\mathcal{M}^{2}_3(0,\mathrm{Id_2}),\zeta_3)$, see
\cite{drmota08}.

We denote the space of probability measures on $\Rset$ by
$\mathcal{M}$ and
\begin{align*}
    \mathcal{M}_3(0,1)
    :=\Big\{\mathcal{L}(Z)\in \mathcal{M}\,\Big|\, \E[|Z|^3]<\infty,
    \E[Z]=0, \V(Z)= 1\Big\}.
\end{align*}
Furthermore, the product of probability measures $\nu_1$ and $\nu_2$
on $\Rset$ by $\nu_1\otimes \nu_2$. Consider the space
\begin{align*}
    \mathcal{G}:=\{ \nu_1 \otimes \mathcal{N}(0,1)\,|\,
    \nu_1 \in \mathcal{M}_3(0,1)\}.
\end{align*}
Then $\mathcal{G}\subset \mathcal{M}^{2}_3(0,\mathrm{Id_2})$.

To show that $(\mathcal{G},\zeta_3)$ is a closed subspace of
$(\mathcal{M}^{2}_3(0,\mathrm{Id_2}),\zeta_3)$, let $(\mu_n\otimes
\mathcal{N}(0,1))_{n\ge 1}$ be a sequence in $\mathcal{G}$ that
converges in $(\mathcal{M}^{2}_3(0,\mathrm{Id_2}),\zeta_3)$, say to
$\mathcal{L}(Y_1,Y_2)$. Since $\zeta_3$-convergence implies weak
convergence, we first obtain that $Y_2$ is standard normally
distributed. Clearly, we have $\mathcal{L}(Y_1)\in
\mathcal{M}_3(0,1)$. Since a weak limit of product measures is a
product measure (see e.g.~\cite[Theorem 2.8(ii)]{bill99}),
$\mathcal{L}(Y_1,Y_2)$ is a product measure. Now
$(\mathcal{G},\zeta_3)$ as a closed subspace of the complete space
$(\mathcal{M}^{2}_3(0,\mathrm{Id_2}),\zeta_3)$ is complete.

We next show that the restriction of $T'_N$ to $\mathcal{G}$ maps to
$\mathcal{G}$. Note that only here do we use the fact that the second
component in the definition of $\mathcal{G}$ is a normal
distribution; see (\ref{rn0116}) below. For $\mu=\mu_1\otimes
\mathcal{N}(0,1) \in \mathcal{G}$, the covariance matrix of
$T'_N(\mu) =: \mathcal{L}(Y_1,Y_2)$ is $\mathrm{Id_2}$ by
\eqref{cov-id}. Since $Y_2$ is distributed as $\sum_{1\le r\le m}
V_r^{1/2}N_r$, where the $N_j$'s are independent normals and
independent of $(V_1,\ldots,V_m)$, we see that $\mathcal{L}(Y_2)
=\mathcal{N}(0,1)$. Thus it remains to show that, for
$T'_N(\mu)\in\mathcal{G}$, the components $Y_1$ and $Y_2$ are
independent. Let $A,B\subset \Rset$ be measurable and
$(Y^{(1)}_1,Y^{(1)}_2), \ldots, (Y^{(m)}_1,Y^{(m)}_2)$ be independent
random vectors that are independent of $(V_1,\ldots,V_m)$ and
identically distributed as $\mu$. Then, denoting the distribution of
$V=(V_1,\ldots,V_m)$ by $\Upsilon$ and, for $v=(v_1,\ldots,v_m)$,
writing $t_N(v):=C_N^{-1/2}(\varphi_m+2\varphi_m^2 \sum_{1\le r\le m}
v_r\log v_r)$, we have
\begin{align}
    \Prob(Y_1\in A, Y_2\in B)
    &=\Prob\left( \sum_{1\le r\le m}
    V_rY^{(1)}_r+t_N(V)\in A,
    \sum_{1\le r\le m} V_r^{1/2}Y^{(2)}_r \in B\right) \nonumber\\
    &=\int \Prob\left( \sum_{1\le r\le m}
    v_rY^{(1)}_r+t_N(v)\in A, \sum_{1\le r\le m}
    v_r^{1/2}Y^{(2)}_r \in B\right) \mathrm{d}\Upsilon(v) \nonumber\\
    &=\int \Prob\left( \sum_{1\le r\le m} v_rY^{(1)}_r
    +t_N(v)\in A\right)\Prob\left(\sum_{1\le r\le m}
    v_r^{1/2}Y^{(2)}_r \in B\right) \mathrm{d}\Upsilon(v) \nonumber\\
    &=\int \Prob\left( \sum_{1\le r\le m} v_rY^{(1)}_r
    +t_N(v)\in A\right)\mathcal{N}(0,1)(B) \mathrm{d}\Upsilon(v)
	\label{rn0116}\\
    &=\Prob(Y_1\in A)\Prob(Y_2\in B).\nonumber
\end{align}
We then deduce that $T'_N(\mu)\in\mathcal{G}$ and $T'_N$ maps
$\mathcal{G}$ to $\mathcal{G}$.

Finally, Banach's fixed point theorem implies that the restriction of
$T'_N$ to $\mathcal{G}$ has a unique fixed point. Since
$\mathcal{G}\subset \mathcal{M}^{2}_3(0,\mathrm{Id_2})$, we find
$\mathcal{L}(X',\Lambda')\in \mathcal{G}$. Consequently, $X'$ and
$\Lambda'$ are independent.
\end{proof}

\paragraph{Proof of Theorem~\ref{thm13}: NPL.} The proof of
Theorem~\ref{thm13} relies on Theorem 4.1 in \cite{neininger04}. The
parameter $d$ there is taken to be the dimension $d=2$ here, and we
choose the parameter $s=3$. Note that the normalization in
(\ref{thmnc}) is as required in \cite[eq.~(22)]{neininger04} and is
identical to the normalization leading to the $Y_n$ in
(\ref{yn_norm}). We need to check the conditions (24)--(26) in
\cite{neininger04}. Condition (24) in our case is, with $F^{(n)}_r$
and $b^{(n)}$ as in (\ref{coeffyn}),
\begin{align*}
    (F^{(n)}_1,\ldots,F^{(n)}_m, b^{(n)})
    \to (F^*_1,\ldots,F^*_m, b_N^*)
\end{align*}
in $L_3$. This is satisfied by (\ref{coeffyn}). Condition (25) in
our case is also satisfied because
\begin{align*}
    \sum_{1\le r\le m} \|F_r^*\|_{\mathrm{op}}^3
    =m \E[V_1^{3/2}]<1.
\end{align*}
Finally, condition (25) is, for all $r=1,\ldots,m$ and all $\ell \in
\N$,
\begin{align*}
    \E\left[\mathbf{1}_{\{I^{(n)}_r\le \ell\}
        \cup \{I^{(n)}_r=n\}}\|F^{(n)}_r\|_{\mathrm{op}}^3\right]
    \to 0 .
\end{align*}
Since $\|F^{(n)}_r\|_{\mathrm{op}}$ are uniformly bounded random
variables, this condition is equivalent to
\begin{align*}
    \Prob\left(I^{(n)}_r\le \ell\right)
    \to 0,
\end{align*}
which is satisfied in view of (\ref{conv_in}). Hence, Theorem 4.1 in
\cite{neininger04} applies and implies the convergence
$\Cov(Q_n)^{-1/2}(Q_n-\E[Q_n]) \to (X',\Lambda')$ in the metric
$\zeta_3$, which implies the stated convergence in distribution.

Note that the components of $T_N'$ imply univariate recursive
distributional equations for $\mathcal{L}(\Lambda')$ and
$\mathcal{L}(X')$:
\begin{align*}
    \Lambda'
    &\stackrel{d}{=} \sum_{1\le r\le m}
    \sqrt{V_r} \Lambda'^{(r)},\\
    X'
    &\stackrel{d}{=} \sum_{1\le r\le m}
    V_r X'^{(r)} + C_N^{-1/2} b_N,
\end{align*}
with conditions on independence and identical distributions
corresponding to the definition of $T_N'$. Moreover, both equations
are subject to the constraints of zero mean, unit variance and
bounded third absolute moment. The solution for
$\mathcal{L}(\Lambda')$ is given by the standard normal distribution,
and a comparison of the equation for $\mathcal{L}(X')$ with
(\ref{def_map_t}) shows that $X'$ is identically distributed as
$C_N^{-1/2} X$ with $X$ as in Theorem \ref{thm12}.

\subsection{Limit law for NPL}\label{univariate}

From the previous two subsections, we see that the limit law of
$(N_n-\E(N_n))/n$ is the unique solution, subject to zero mean and
finite variance, of the recursive distributional equation
\[
    X
    \stackrel{d}{=} \sum_{1\le r\le m} V_r X^{(r)} + \phi
        +2\phi^2\sum_{1\le r\le m} V_r\log V_r,
\]
where $X^{(1)},\ldots,X^{(m)},V$ are independent and the $X^{(r)}$
have the same distribution as $X$.

Moreover, it is well-known (see Corollary 5.2 in \cite{neininger99})
that the limit law of $(K_n-\E(K_n))/n$, which we denote by
$\mathcal{L}(K)$ in Section \ref{known-results}, is the unique
solution, again subject to zero mean and finite variance, of
\begin{align}\label{rn1212}
    X
    \stackrel{d}{=} \sum_{1\le r\le m} V_r X^{(r)} + 1
        +2\phi\sum_{1\le r\le m} V_r\log V_r,
\end{align}
where the meaning of the notations is as above.

Comparing these two distributional recurrences, we see that the
solution to the first one is $\mathcal{L}(\phi K)$. Thus, we have
\[
    \frac{N_n-\E(N_n)}{n}\stackrel{d}{\longrightarrow}\phi K,
\]
i.e., the limit law of $K_n$ and $N_n$ are up to a constant
identical. In fact, if one is only interested in this result, then
one does not need the analysis in the last two subsections but there
are simpler approaches, as we discussed below.

\subsection{Short proofs for the limit law of $N_n$}
\label{app-univariate}

In this section, we discuss different means of proving directly the
limit law for NPL without the detour via the bivariate setting from
Sections \ref{secns<=26} and \ref{secns>26}.

\paragraph{Limit law for NPL by the contraction method.}

A first alternative approach to the limit law for NPL uses the
contraction method and ``over-normalizing'' in recurrence
(\ref{rec1}). More precisely, normalize with an $\alpha<\alpha'<1$ by
\begin{align*}
    \mathcal{R}_n:=\left[
    \begin{array}{cc} n{-1} & 0\\ 0& n^{-\alpha'}\end{array}\right]
    \left(\!\!\begin{array}{c} N_n-\E[N_n]
    \\ S_n-\E[S_n]\end{array}\!\!\right), \qquad (n\ge 1).
\end{align*}
Now the recurrence (\ref{rec1}) leads to the limit equation
\begin{align}\label{rn_nerec}
    \left(\mathcal{R}\right)^{\mathrm{t}}
    \stackrel{d}{=}\sum_{1\le r\le m} \left[
    \begin{array}{cc}
        V_r & 0\\ 0& V_r^{\alpha'}
    \end{array}\right]
    \left(\mathcal{R}^{(r)}\right)^{\mathrm{t}}
    + \left(\!\!
    \begin{array}{c}
        b_N \\ 0
    \end{array}\!\!\right),
\end{align}
with conditions on independence and identical distributions as in
(\ref{def_map_t}). Theorem 4.1 in \cite{neininger01} directly applies
and implies that $\mathcal{R}_n \to \mathcal{R}$ in distribution and
with second (mixed) moments, where $\mathcal{R}$ is the unique fixed
point subject to zero mean and finite second moment of the recursive
distributional equation (\ref{rn_nerec}). By substituting into
(\ref{rn_nerec}), we see that $(\phi K,0)$ has the distribution of
$\mathcal{R}$, which implies that
\begin{align*}
    \frac{N_n-\E[N_n]}{n}
    \stackrel{d}{\longrightarrow} \phi K.
\end{align*}

\paragraph{Univariate limit law for NPL via Slutsky's theorem.}
Another approach is to apply Slutsky's theorem. For that purpose,
we consider the moment generating function
\[
    \bar{P}_n(u,v,w)
    =\E\left(e^{\bar{S}_n u
        +\bar{\SK}_n v+\bar{\SN}_n w}\right).
\]
Then $\bar{P}_n$ satisfies the recurrence
\[
    \bar{P}_n(u,v,w)
    =\frac{1}{\binom{n}{m-1}}
        \sum_{\mathbf{j}}\bar{P}_{j_1}(u+w,v,w)
        \cdots\bar{P}_{j_l}(u+w,v,w)e^{\Delta_{\mathbf{j}} u
        +\nabla_{\mathbf{j}} v+\delta_{\mathbf{j}} w},
\]
with the initial conditions $P_n(u,v,w)=1$ for $0\le n\le m-2$.
Now define
\[
    V_n^{[KN]}
    :=\Cov(\SK_n,\SN_n).
\]
Then
\[
    V_n^{[KN]}
    =m\sum_{0\le j\le n-m+1}\pi_{n,j}V_j^{[KN]}+b_n^{[KN]},
\]
where
\[
    b_n^{[KN]}
    =\frac{1}{\binom{n}{m-1}}\sum_{\mathbf{j}}
        \left(V_j^{[SK]}+\nabla_{\mathbf{j}}\delta_{\mathbf{j}}\right)
    =V_n^{[SK]}+\frac{1}{\binom{n}{m-1}}\sum_{\mathbf{j}}
        \left(\nabla_{\mathbf{j}}\delta_{\mathbf{j}}
        -\Delta_{\mathbf{j}}\nabla_{\mathbf{j}}\right).
\]
Observe that Lemma \ref{est-dn} and Lemma \ref{est-dn}, together with
the asymptotics of $V_n^{[SK]}$, imply that
\[
    b_n^{[KN]}
    \sim\frac{1}{\binom{n}{m-1}}
        \sum_{\mathbf{j}}\nabla_{\mathbf{j}}\delta_{\mathbf{j}}
    \sim\frac{\phi}{\binom{n}{m-1}}
        \sum_{\mathbf{j}}\nabla_{\mathbf{j}}^2
    \sim\phi b_n^{[\SK]}.
\]
Consequently, by the same method of proofs used in
Section~\ref{sec_spl}, we see that
\[
    V_n^{[KN]}
    \sim\phi C_\SK n^2.
\]
Now consider the difference
\begin{align*}
    \E(\phi\bar{\SK}_n-\bar{\SN}_n)^2
    &=\phi^2V_n^{[K]}-2\phi V_n^{[KN]}+V^{[N]}_n\\
    &\sim\phi^2C_{\SK}n^2-2\phi^2C_{\SK}n^2
        +\phi^2C_{\SK}n^2\\
    &=o(n^2).
\end{align*}
Consequently, by Chebyshev's inequality, we obtain the convergence in
probability
\[
    \frac{\phi\bar{\SK}_n-\bar{\SN}_n}{n}
    \stackrel{\Prob}\longrightarrow 0.
\]
From this, the claimed result follows from Slutsky's theorem and the
limit law for KPL.

Note that this argument in addition gives the following consequence.
\begin{cor}\label{cor_asycorrel} The correlation coefficient between
$K_n$ and $N_n$ tends asymptotically to one
\[
    \rho(\SK_n,\SN_n)\to 1.
\]
\end{cor}

\paragraph{Identical limit random variables.} To the pair
$(N_n,K_n)$, we could as well apply the contraction method, and prove
that the normalization $(N_n-\E(N_n))/n, (K_n-\E(K_n))/n)$ converges
to a limit given by
\begin{align*}
    \left(\mathcal{P}\right)^{\mathrm{t}}
    \stackrel{d}{=}\sum_{1\le r\le m}\left[
    \begin{array}{cc}
        V_r & 0\\
        0& V_r
    \end{array}\right]
    \left(\mathcal{P}^{(r)}\right)^{\mathrm{t}} +\left(\!\!
    \begin{array}{c}
        \phi b_K \\
        b_K
    \end{array}\!\!\right),
\end{align*}
with conditions on independence and identical distributions as in
(\ref{def_map_t}) and subject to zero mean and finite second moment.
By plugging in, we find that $(\phi K, K)$ has the limit
distribution. This re-derives Corollary \ref{cor_asycorrel} and shows
that the limit random variables (up to scaling) are even almost
surely identical. It seems reasonable to conjecture that the sequences
\begin{align*}
    \left(\frac{N_n-\E[N_n]}{\phi n}\right)_{n\ge 1},
    \qquad \left(\frac{K_n-\E[K_n]}{n}\right)_{n\ge 1}
\end{align*}
both convergence almost surely to the same random variable with the
distribution of $K$. This requires the $m$-ary search trees to grow
as a combinatorial Markov chain, which canonically is obtained by
building up the tree from i.i.d.~uniformly on $[0,1]$ distributed
data. For the notion of a combinatorial Markov chain and related
results on binary search trees, see Gr\"ubel \cite{grubel14}.

\section{Extensions}
\label{sec_ext}

The dependence and phase changes we established above for space
requirement and path lengths in random $m$-ary search trees are not
confined to these shape parameters, neither are they specific to
$m$-ary search trees. The same study (including the same methods of
proof) can be carried out for other shape parameters and other
classes of random trees. We consider first random median-of-$(2t+1)$
search trees in this section, where we discuss the joint asymptotics
of size (defined as the number of nodes with at least $2t$
descendants) and total key path length (which is also the major cost
measure for Quicksort using the median-of-$(2t+1)$ technique). Random
quadtrees will be also briefly discussed. Then we consider another
line of extension, namely, to other shape parameters in these trees.
Since the technicalities follow more or less the same pattern, we
skip all proofs.

\subsection{Random fringe-balanced binary search trees}

Fringe-balanced binary search trees (FBBSTs) are binary search trees
($m=2$) with local re-organizations for all subtrees of size exactly
$2t+1$ into more balanced ones. In terms of quicksort, the
corresponding tree structures choose at each partitioning stage the
median of a sample of $2t+1$ elements to partition the elements into
smaller and larger groups. For a precise description and other
connections, see \cite{chern01,devroye93}. The number of nodes
$\mathcal{S}_n$ with at least $2t$ descendants (or the number of
median-partitioning stages) and the total path length of these nodes
(TPL; KPL$=$NPL for binary search trees) $\mathcal{X}_n$ of a random
FBBST constructed from a random permutation of $n$ elements satisfy
the following distributional recurrence ($\mathcal{Q}_n
:=(\mathcal{X}_n,\mathcal{S}_n)$)
\begin{align*}
    \left(\mathcal{Q}_n\right)^{\mathrm{t}}
    \stackrel{d}{=}
        \left(\mathcal{Q}^{(1)}_{I'_n}\right)^{\mathrm{t}} +
        \left(\mathcal{Q}^{(2)}_{n-1-I'_n}\right)^{\mathrm{t}}
        +\left(\!\!
        \begin{array}{c}
            n-1 \\ 1
        \end{array}\!\!\right), \qquad(n\ge 2t+1),
\end{align*}
with conditions on independence and identical distributions as in
(\ref{rec1}) and the initial conditions $\mathcal{S}_0=\cdots={\cal
S}_{2t}= \mathcal{X}_0=\cdots=\mathcal{X}_{2t}=0$. Here
\[
    \Prob(I'_n=j)
    =\frac{\binom{j-1}{t}\binom{n-j}{t}}{\binom{n}{2t+1}}
    \qquad(t\le j\le n-1-t).
\]

We start with the mean. First, for $\mathcal{S}_n$, it was proved in
\cite{chern01} that
\begin{align}
    \E(\mathcal{S}_n)
    =C_1(n+1)-1
        +\sum_{2\le k\le 3}\frac{C_k}{\Gamma(\varrho_k)}\,
        n^{\varrho_k-1}+o(n^{\alpha_t -1})
    \label{def_omega0}
\end{align}
where
\[
    C_k
    =\frac{t!}{2(\varrho_k-1)\varrho_k
        \cdots(\varrho_k+t-1)\sum_{t\le j\le 2t}
        \frac1{j+\varrho_k}} \qquad(k=1,\dots,t+1),
\]
with $\varrho_1=2>\Re(\varrho_2)=\Re(\varrho_3)=\alpha_t
>\Re(\varrho_4)\ge \cdots\ge\Re(\varrho_{t+1})$ being the zeros of
the indicial equation
\[
    (z+t)\cdots (z+2t)-\frac{2(2t+1)!}{t!}.
\]
In particular,
\[
    C_1 = \phi_t := \frac1{2(t+1)(H_{2t+2}-H_{t+1})}.
\]

Moreover, using the transfer theorems from \cite{chern01}, we obtain,
for the mean of $\mathcal{X}_n$,
\[
    \E(\mathcal{X}_n)
    =\frac{1}{H_{2t+2}-H_{t+1}}\,n\log n + c_t n +o(n),
\]
for some constant $c_t$. The same method of proofs (asymptotic
transfer and the approach used in Section~\ref{sec_corrmary}) also
leads to asymptotic estimates for the variances and the covariance
between $\mathcal{X}_n$ and $\mathcal{S}_n$.

\begin{thm} The variance of the number of non-leaf nodes
$\mathcal{S}_n$ and that of the TPL $\mathcal{X}_n$ in a random
FBBST, and their covariance satisfy
\begin{align*}
    \V(\mathcal{S}_n)
    &\sim
    \begin{cases}
        D_Sn,&\text{if}\ 1\le t\le 58;\\
        G_1(\beta_t\log n)n^{2\alpha_t -2},&\text{if}\ t\ge 59,
    \end{cases}\\
    \Cov(\mathcal{S}_n,\mathcal{X}_n)
    &\sim
    \begin{cases}
        D_Rn,&\text{if}\ 1\le t\le 28;\\
        G_2\left(\beta_t\log n\right)n^{\alpha_t },
            &\text{if}\ t\ge 29,
    \end{cases}\\
    \V(\mathcal{X}_n)
    &\sim D_Xn^2,
\end{align*}
where $D_S,D_R$ are suitable constants, $\beta_t=\Im(\varrho_2)$,
and all other constants and functions are given below.
\end{thm}

The periodic functions in the above theorem are given by
\begin{align*}
    G_1(z)
    &=2\frac{\vert C_2\vert^2}{\vert\Gamma(\varrho_2)\vert^2}
        \left(-1+\frac{2(2t+1)!\vert\Gamma(\varrho_2+t)\vert^2}
        {t!^2\Gamma(2\alpha_t +2t)-2t!(2t+1)!
        \Gamma(2\alpha_t +t-1)}\right)\\
    &\qquad\quad+2\Re\left(\frac{C_2^2e^{2i z}}
        {\Gamma(\varrho_2)^2}\left(-1+
        2\frac{2(2t+1)!\Gamma(\varrho_2+t)^2}
        {t!^2\Gamma(2\varrho_2+2t)
        -2t!(2t+1)!\Gamma(2\varrho_2+t-1)}\right)\right)
\end{align*}
and
\begin{align*}
    G_2(z)
    &= \Re\Bigg(\frac{C_2e^{i z}}
        {\Gamma(\varrho_2)}\Bigg(\frac{\varrho_2+2t+1}{t+1} \\
        &\quad -\frac{(\varrho_2+2t+1)\psi(\varrho_2+2t+2)
        -(\varrho_2+t)\psi(\varrho_2+t+1)-(t+1)(H_{t+1}-\gamma)}
        {(t+1)(H_{2t+2}-H_{t+1})}\Bigg)\Bigg),
\end{align*}
respectively. Moreover, we have
\begin{align*}
    D_X
    &=\frac{1}{(H_{2t+2}-H_{t+1})^2}
        \Bigg(\frac{2t+3}{t+1}\,H_{2t+2}^{(2)}
        -\frac{t+2}{t+1}\,H_{t+1}^{(2)}
        -\frac{\pi^2}{6}\Bigg).
\end{align*}
The limit law for the normalized TPL of random FBBSTs was first shown
in the dissertation of Bruhn, \cite{bruhn96}; see also
\cite{broutin12, chern01, munsonius11,rosler01}. The phase change of
the limit law of the normalized $\mathcal{S}_n$ was first discovered
in \cite{chern01}.

To describe the joint limiting behavior of $\mathcal{S}_n$ and
$\mathcal{X}_n$, we denote by $\mathcal{V}$ a random variable that is
the median of $(2t+1)$ independent, identically distributed uniform
$[0,1]$ random variables, i.e., a Beta$(t+1,t+1)$ distribution. We
define the map $T_{\mathrm{med}}$ by
\begin{align*}
    T_{\mathrm{med}}: \mathcal{M}^{\R \times \C}
    &\to \mathcal{M}^{\R \times \C},  \nonumber\\
    \mathcal{L}(Z,W)
    &\mapsto \mathcal{L}\left(\left[\!\!
    \begin{array}{cc}
        \mathcal{V}& 0\\
        0& \mathcal{V}^{\varrho_2}
    \end{array}\!\!\right]\left(\!\!
    \begin{array}{c}
        Z^{(1)} \\ W^{(1)}
    \end{array}\!\!\right) +\left[\!\!
    \begin{array}{cc}
        1-\mathcal{V} & 0\\
        0& (1-\mathcal{V})^{\varrho_2}
    \end{array}\!\!\right]\left(\!\!
    \begin{array}{c}
        Z^{(2)} \\ W^{(2)}
    \end{array}\!\!\right)+ \left(\!\!
    \begin{array}{c}
        b_M \\ 0
    \end{array}\!\!\right) \right),
\end{align*}
with conditions on independence and distributions as in
(\ref{def_map_t}) and
\begin{align*}
    b_M
    := 1+\frac{1}{H_{2t+2}-H_{t+1}}
        \left(\mathcal{V}\log \mathcal{V}
        + (1-\mathcal{V})\log(1-\mathcal{V})\right).
\end{align*}
Then Lemma \ref{fplem1} and its proof also apply to the map
$T_{\mathrm{med}}$ as long as $t\ge 59$. The normalization used is
given by
\begin{align}\label{def_calyn}
    \mathcal{Y}_n
    :=\left(\frac{\mathcal{X}_n-
        \E(\mathcal{X}_n)}{n}, \frac{\mathcal{S}_n-C_1n}
        {n^{\alpha_t-1}}\right),\qquad (n\ge 1).
\end{align}
We have the following asymptotic behavior for $t\ge 59$. Rewrite
\eqref{def_omega0} as
\begin{align}
    \E(\mathcal{S}_n)
    &=C_1(n+1)-1+ \Re(\vartheta n^{\varrho_2})
        +o(n^{\alpha_t -1}),
    \label{def_omega}
\end{align}
where $\vartheta := 2\Re(C_2/\Gamma(\varrho_2))$.

\begin{thm} Assume $t\ge 59$. Let $\mathcal{Y}_n$ be the
normalization of TPL and the number of non-leaf nodes in a random
FBBST defined in (\ref{def_calyn}). Denote by ${\cal
L}(X_{\mathrm{med}}, \Lambda_{\mathrm{med}})$ the unique fixed point
of the restriction of $T_{\mathrm{med}}$ to $\mathcal{M}^{\R \times
\C}_2((0,\vartheta))$ with $\vartheta$ defined in (\ref{def_omega}).
Then, denoting by $\beta_t:=\Im(\varrho_2)$, we have
\begin{align*}
    \ell_2\left(\mathcal{Y}_n,  (X_{\mathrm{med}},\Re(n^{i\beta_t}
    \Lambda_{\mathrm{med}}))\right) \to 0, \qquad (n\to\infty).
\end{align*}
\end{thm}

For the range of $1\le t \le 58$, we define $b_{\mathrm{med}}^*
:=(D_X^{-1/2}b_M, 0)^{\mathrm{t}}$ and the map $T_{\mathrm{med}}'$ on
$\mathcal{M}^{2}$:
\begin{align*}
    T'_{\mathrm{med}}:\mathcal{M}^{2}
    &\to \mathcal{M}^{2}, \\
    \mathcal{L}(Z,W)
    &\mapsto \mathcal{L}\left(\left[\!\!
    \begin{array}{cc}
        \mathcal{V}& 0\\ 0& \mathcal{V}^{1/2}
    \end{array}\!\!\right]\left(\!\!
    \begin{array}{c}
        Z^{(1)} \\ W^{(1)}
    \end{array}\!\!\right) +\left[\!\!
    \begin{array}{cc}
        1-\mathcal{V} & 0\\
        0& (1-\mathcal{V})^{1/2}
    \end{array}\!\!\right]\left(\!\!
    \begin{array}{c}
        Z^{(2)} \\ W^{(2)}
    \end{array}\!\!\right)
    + b_{\mathrm{med}}^* \right),
\end{align*}
with conditions on independence and distributions as in
(\ref{def_map_t'}). Again Lemma \ref{fpm23} and its proof apply to
$T_{\mathrm{med}}'$ and imply that the restriction of
$T_{\mathrm{med}}'$ to $\mathcal{M}^{2}_3(0,\mathrm{Id_2})$ has a
unique fixed point $\mathcal{L}(X_{\mathrm{med}}',
\Lambda_{\mathrm{med}}')$.

Similar to the small $m$ case of $m$-ary search trees, the remaining
range $1\le t\le 58$ also leads to a convergence in distribution.

\begin{thm} Assume $1\le t\le 58$. Let $\mathcal{Q}_n
=(\mathcal{X}_n, \mathcal{S}_n)$ be the vector of TPL and the number
of non-leaf nodes in a random FBBST. With $\mathcal{L}
(X_{\mathrm{med}}', \Lambda_{\mathrm{med}}')$ as above, we have
\begin{align*}
    \Cov(\mathcal{Q}_n)^{-1/2}
	\left(\mathcal{Q}_n-\E[\mathcal{Q}_n]\right)
    \stackrel{d}{\longrightarrow}
    \mathcal{L}\left(X_{\mathrm{med}}',
	\Lambda_{\mathrm{med}}'\right),
\end{align*}
where $\Lambda_{\mathrm{med}}'$ is a standard normal distribution.
Moreover, $X_{\mathrm{med}}'$ and $\Lambda_{\mathrm{med}}'$ are
independent.
\end{thm}

\subsection{Random quadtrees}

Point quadtrees, first proposed by Finkel and Bentley
\cite{finkel74}, are one of the most natural extensions of binary
search trees to multivariate data in which each point splits the
$d$-dimensional space into $2^d$ subspaces, corresponding to $2^d$
subtrees in the corresponding tree structure. For a precise
definition of random $d$-dimensional quadtrees; see
\cite{chern07,mahmoud92}. Since the space requirement is a constant,
we discuss the number of leaves $L_n$ and the internal path length
$\Xi_n$ in this section. Note that for the pair $\mathcal{W}_n
:=(\Xi_n,L_n)$, we have, for all $n\ge 2$,
\begin{align*}
    \left(\mathcal{W}_n\right)^{\mathrm{t}}
    \stackrel{d}{=} \sum_{1\le r\le 2^d}
        \left(\mathcal{W}^{(r)}_{J_r}\right)^{\mathrm{t}}
        + \left(\!\!
        \begin{array}{c} n-1 \\ 0 \end{array}\!\!\right),
\end{align*}
with conditions on independence and identical distributions as in
(\ref{rec1}), where the initial conditions are $L_0=0,L_1=1,
\Xi_0=\Xi_1=0$. Moreover, the underlying splitting probabilities are
given by
\[
    \Prob(J_1=j_1,\ldots,J_{2^d}=j_{2^d})
    =\binom{n-1}{j_1,\ldots,j_{2^d}}
        \int_{[0,1]^d}q_1(\mathbf{x})^{j_1}\cdots
        q_{2^d}(\mathbf{x})^{j_{2^d}}\mathrm{d}\mathbf{x},
\]
where $j_1,\ldots,j_{2^d}\ge 0, j_1+\cdots+j_{2^d}=n-1,
\mathbf{x}=(x_1,\ldots,x_d)$ and
\[
    q_{h}(\mathbf{x})
    =\prod_{1\le l\le d}\left((1-b_l)x_l+b_l(1-x_l)\right),
\]
with $(b_1,\ldots,b_d)_{2}$ being the binary representation of
$h-1$.

First, it was proved in \cite{chern07} that the mean of $L_n$
satisfies, for $d\ge 2$,
\begin{align}
    \E(L_n)
    &=\chi_dn+c_{+}n^{\hat{\alpha}+i\hat{\beta}}
        +c_{-}n^{\hat{\alpha}-i\hat{\beta}}+\frac{\chi_d}{2^d-1}
        +o(n^{\hat{\alpha}}),
    \label{def_hatomega0}
\end{align}
where $\chi_d,c_{+},c_{-}$ (which is the conjugate of $c_+$) are
given in \cite{chern07}, and $2e^{2\pi i/d}= \hat{\alpha}
+1+i\hat{\beta}$. Moreover, the asymptotic transfer results in
\cite{chern07} also lead to the asymptotic approximation (see also
\cite{flajolet95})
\[
    \E(\Xi_n)
    =\frac{2}{d}\,n\log n +\hat{c}n +o(n),
\]
for some explicitly computable constant $\hat{c}$. In a similar
manner, we can characterize the asymptotics of the variances and the
covariance.

\begin{thm} For the number of leaves $L_n$ and the internal path
length $\Xi_n$ in random $d$-dimensional quadtrees, we have
\begin{align*}
    \V(L_n)
    &\sim \begin{cases}
        E_Ln,&\text{if}\ 1\le d\le 8;\\
        P_1\bigl(\hat\beta\log n\bigr)
		n^{2\hat{\alpha}},&\text{if}\ d\ge 9,
    \end{cases}\\
    \Cov(\Xi_n,L_n)
    &\sim \begin{cases}
        E_Rn,&\text{if}\ 1\le d\le 5;\\
        P_2\bigl(\hat\beta\log n\bigr)n^{\hat{\alpha}+1},
            &\text{if}\ d\ge 6,
    \end{cases}\\
    \V(\Xi_n)&\sim E_Xn^2,
\end{align*}
where $E_L,E_R$ are suitable constants, $\hat\beta:=
2\sin\frac{2\pi}d$, and all other constants and functions are given
below.
\end{thm}

The periodic functions above are given by
\begin{align*}
    P_1(z) &=2\frac{(2\hat{\alpha}+1)^d}{(2\hat{\alpha}+1)^d-2^d}
    \vert c_{+}\vert^2c_L(\hat{\alpha}
    +i\hat{\beta},\hat{\alpha}-i\hat{\beta})\\
    &\qquad\qquad+2\Re\left(\frac{(2\hat{\alpha}+2i\hat{\beta}+1)^d}
    {(2\hat{\alpha}+2i\hat{\beta}+1)^d-2^d}
    c_{+}^2c_L(\hat{\alpha}+i\hat{\beta},\hat{\alpha}+i\hat{\beta})
    e^{2iz}\right),
\end{align*}
where $c_L(u,v)=1-\eta(0,u)-\eta(0,v)+2^d\eta(u,v)$ with
\[
    \eta(u,v)
    :=\left(\frac{1}{u+v+1}+\frac{\Gamma(u+1)\Gamma(v+1)}
        {\Gamma(u+v+2)}\right)^d
\]
and
\[
    P_2(z)
    =2\Re\left(\frac{(\hat{\alpha}+i\hat{\beta}+2)^d}
        {(\hat{\alpha}+i\hat{\beta}+2)^d-2^d}\,c_+
        c_{K}(\hat{\alpha}+i\hat{\beta})e^{iz}\right),
\]
where
\[
    c_K(u,v)
    =\eta(0,u)+\frac{2^{d+1}}{d}
        \frac{\partial}{\partial v}\eta(u,v)\Big\vert_{v=1}.
\]
Finally,
\[
    E_{X}
    =\frac{3^d}{3^d-2^d}\cdot\frac{21-2\pi^2}{9d}.
\]
The limit law for the normalized internal path length of random
$d$-dimensional quadtrees was first obtained in \cite{neininger99};
see also \cite{broutin12,chern07,munsonius11}. The asymptotic
behavior of the normalized number of leaves together with its phase
change was first discovered in \cite{chern07}; see also
\cite{dean02,janson04,janson06b,janson08} for closely related types
of phase changes.

We now describe the joint behavior of $\Xi_n$ and $L_n$. A random
variable $U$ uniformly distributed over the unit hypercube $[0,1]^d$
decomposes this cube into $2^d$ quadrants by drawing the $d$
hyperplanes through $U$ perpendicular to the edges of the cube.
Choose an ordering of these quadrants and denote their volumes by
$\langle U\rangle_1,\ldots, \langle U\rangle_{2^d}$; see
\cite[Section 2]{neininger99}. Now define the map $T_{\mathrm{quad}}$
by (with $\delta_2:=2e^{2\pi i/d}$)
\begin{align*}%\label{def_map_q}
    T_{\mathrm{quad}}: \mathcal{M}^{\R \times \C}
    &\to \mathcal{M}^{\R \times \C},  \\
    \mathcal{L}(Z,W)
    &\mapsto \mathcal{L}\left(\sum_{1\le r\le 2^d}
    \left[\!\!
    \begin{array}{cc}
        \langle U\rangle_r& 0\\
        0& \langle U\rangle_r^{\delta_2}
    \end{array}\!\!\right]\left(\!\!
    \begin{array}{c}
        Z^{(r)} \\
        W^{(r)}
    \end{array}\!\!\right)+ \left(\!\!
    \begin{array}{c}
        b_Q \\ 0
    \end{array}\!\!\right) \right),
\end{align*}
with conditions on independence and distributions as in
(\ref{def_map_t}), and
\begin{align*}
    b_Q
    := 1+\frac{2}{d}\sum_{1\le r\le 2^d}
        \langle U\rangle_r \log \langle U\rangle_r.
\end{align*}
Then Lemma \ref{fplem1} and its proof also apply to map
$T_{\mathrm{quad}}$ as long as $d\ge 9$. The normalization used is
given by
\begin{align}\label{def_calvn}
    \mathcal{V}_n
    :=\left(
        \frac{\Xi_n-\E(\Xi_n)}{n},
        \frac{L_n-\chi_dn}{n^{\hat{\alpha}}}
    \right) \qquad (n\ge 1).
\end{align}
Rewrite \eqref{def_hatomega0} as
\begin{align}
    \E(L_n)
    &=\chi_dn +\Re(\hat{\vartheta}n^{\hat{\alpha}+i\hat{\beta}})+
        \frac{\chi_d}{2^d-1}+o(n^{\hat{\alpha}}),
    \label{def_hatomega}
\end{align}
where $\hat{\vartheta} = 2c_+$.

\begin{thm} Assume $d\ge 9$. Let $\mathcal{V}_n$ denote the
normalization of the internal path length and the number of leaves in
a random $d$-dimensional quadtree defined in (\ref{def_calvn}).
Denote by $\mathcal{L}(X_{\mathrm{quad}}, \Lambda_{\mathrm{quad}})$
the unique fixed point of the restriction of $T_{\mathrm{quad}}$ to
$\mathcal{M}^{\R \times \C}_2((0,\hat{\vartheta}))$ with
$\hat{\vartheta}$ defined in (\ref{def_hatomega}). Then we have
\begin{align*}
    \ell_2\left(
        \mathcal{V}_n,
        \bigl(X_{\mathrm{quad}},
        \Re\bigl(n^{i\hat{\beta}}
		\Lambda_{\mathrm{quad}}\bigr)\bigr)
    \right)
    \to 0.
\end{align*}
\end{thm}

For the remaining range of $1\le d \le 8$, we define
$b_{\mathrm{quad}}^* :=(E_X^{-1/2}b_M,0)^{\mathrm{t}}$ and the map
$T_{\mathrm{quad}}'$ on $\mathcal{M}^{2}$
\begin{align*}
    T'_{\mathrm{quad}}: \mathcal{M}^{2}
    &\to \mathcal{M}^{2}, \\
    \mathcal{L}(Z,W)
    &\mapsto \mathcal{L}\left(\sum_{1\le r\le 2^d}\left[\!\!
    \begin{array}{cc}
        \langle U\rangle_r& 0\\
        0& \langle U\rangle_r^{1/2}
    \end{array}\!\!\right] \left(\!\!
    \begin{array}{c}
        Z^{(r)} \\ W^{(r)}
    \end{array}\!\!\right) + b_{\mathrm{quad}}^*\right),
\end{align*}
with conditions on independence and distributions as in
(\ref{def_map_t'}). Similarly, Lemma \ref{fpm23} and its proof again
apply to $T_{\mathrm{quad}}'$ and imply that the restriction of
$T_{\mathrm{quad}}'$ to $\mathcal{M}^{2}_3(0,\mathrm{Id_2})$ has a
unique fixed point ${\cal L} (X_{\mathrm{quad}}',
\Lambda_{\mathrm{quad}}')$.

\begin{thm}\label{thm13m} Assume $1\le d\le 8$. Let
$\mathcal{V}_n=(\Xi_n,L_n)$ denote the vector of internal path length
and the number of leaves in a random $d$-dimensional quadtree. With
${\cal L}(X_{\mathrm{quad}}',\Lambda_{\mathrm{quad}}')$ as above, we
have
\begin{align*}
    \Cov(\mathcal{V}_n)^{-1/2}(\mathcal{V}_n-\E[\mathcal{V}_n])
    \stackrel{d}{\longrightarrow}
    \mathcal{L}\bigl(X_{\mathrm{quad}}',
	\Lambda_{\mathrm{quad}}'\bigr),
\end{align*}%
where $\Lambda_{\mathrm{quad}}'$ is a standard normal distribution,
and $X_{\mathrm{quad}}'$, and $\Lambda_{\mathrm{quad}}'$ are
independent.
\end{thm}

The case when $d=1$ corresponds to binary search trees, or
equivalently, to Hoare's quicksort, and the above theorem can be
re-worded as follows. \emph{The number of comparisons and the number
of partitioning stages used by Hoare's quicksort are asymptotically
uncorrelated and independent.} Note that our results in the previous
section for random FBBSTs give indeed a stronger statement for the
asymptotic independence or asymptotical periodicity for quicksort
using median-of-($2t+1$).

\subsection{More general shape parameters}

Our study can be extended to other shape parameters. For random
$m$-ary search trees, the generality of Proposition~\ref{asymp-trans}
provides an effective means of widening our study to a broader class
of ``toll functions" in the definitions of $S_n$, $K_n$ and $N_n$.
For example, the following extensions are straightforward.
\begin{align}\label{Sn-extdd}
    S_n\stackrel{d}{=}S_{I_1}^{(1)}+\cdots+S_{I_m}^{(m)}
    +\begin{cases}c+o(n^{-\varepsilon}),
    &\text{if}\ 2\le m\le 13;\\
    o(n^{\alpha-1}),&\text{if}\ m\ge 14
    \end{cases}
\end{align}
for some constant $c$, and

\begin{description}

\item[--] $K_n\stackrel{d}{=}K_{I_1}^{(1)}+\cdots
+K_{I_m}^{(m)}+n+t_n$ with
\begin{align}\label{tn-iff}
    t_n=o(n) \quad \text{and}\quad
    \left\vert\sum_{n}t_n n^{-2}\right\vert<\infty,
\end{align}
and

\item[--] $N_n\stackrel{d}{=}N_{I_1}^{(1)}+\cdots+N_{I_m}^{(m)}
+S_{I_1}^{(1)}+\cdots+S_{I_m}^{(m)}+t_n$, where the $S_n$'s satisfy
\eqref{Sn-extdd} and $t_n$ satisfies \eqref{tn-iff}.

\end{description}

Because the same iff-condition \eqref{tn-iff} also appears in the
recurrence relations arising from the two other classes of random
trees (see \cite{chern07,chern01}), exactly the same conditions can
be used to extend the consideration for FBBSTs and quadtrees. Details
are omitted here.

\bibliographystyle{abbrv}
\bibliography{m-ary-search-tree}

\begin{thebibliography}{10}

\bibitem{baeza-yates87}
R.~A. Baeza-Yates.
\newblock Some average measures in {$m$}-ary search trees.
\newblock {\em Inform. Process. Lett.}, 25(6):375--381, 1987.

\bibitem{bill99}
P.~Billingsley.
\newblock {\em Convergence of probability measures}.
\newblock Wiley Series in Probability and Statistics: Probability and
  Statistics. John Wiley \& Sons, Inc., New York, second edition, 1999.
\newblock A Wiley-Interscience Publication.

\bibitem{bindjeme12}
P.~Bindjeme and J.~A. Fill.
\newblock Exact $l^2$-distance from the limit for quicksort key comparisons
  (extended abstract).
\newblock {\em Discrete Math. Theor. Comput. Sci., Nancy}, pages 339--348,
  2012.

\bibitem{broutin12}
N.~Broutin and C.~Holmgren.
\newblock The total path length of split trees.
\newblock {\em Ann. Appl. Probab}, 22:1745--1777, 2012.

\bibitem{bruhn96}
V.~Bruhn.
\newblock {\em Eine Methode zur asymptotischen Behandlung einer Klasse von
  Rekursionsgleichungen mit einer Anwendung in der stochastischen Analyse des
  Quicksort-Algorithmus}.
\newblock PhD thesis, Christian-Albrechts-Universit{\"a}t zu Kiel Dissertation,
  1996.

\bibitem{chauvin04}
B.~Chauvin and N.~Pouyanne.
\newblock {$m$}-ary search trees when {$m\ge27$}: a strong asymptotics for the
  space requirements.
\newblock {\em Random Structures Algorithms}, 24(2):133--154, 2004.

\bibitem{chern07}
H.-H. Chern, M.~Fuchs, and H.-K. Hwang.
\newblock Phase changes in random point quadtrees.
\newblock {\em ACM Trans. Algorithms}, 3(2):Art. 12, 51, 2007.

\bibitem{chern01}
H.-H. Chern and H.-K. Hwang.
\newblock Phase changes in random {$m$}-ary search trees and generalized
  quicksort.
\newblock {\em Random Structures Algorithms}, 19(3-4):316--358, 2001.

\bibitem{dean02}
D.~S. Dean and S.~N. Majumdar.
\newblock Phase transition in a random fragmentation problem with applications
  to computer science.
\newblock {\em J. Phys. A}, 35(32):L501--L507, 2002.

\bibitem{devroye93}
L.~Devroye.
\newblock On the expected height of fringe-balanced trees.
\newblock {\em Acta Inform.}, 30(5):459--466, 1993.

\bibitem{drmota08}
M.~Drmota, S.~Janson, and R.~Neininger.
\newblock A functional limit theorem for the profile of search trees.
\newblock {\em Ann. Appl. Probab.}, 18(1):288--333, 2008.

\bibitem{fill13}
J.~A. Fill.
\newblock Distributional convergence for the number of symbol comparisons used
  by quicksort.
\newblock {\em Ann. Appl. Probab}, 23:1129--1147, 2013.

\bibitem{fill04}
J.~A. Fill and N.~Kapur.
\newblock The space requirement of $m$-ary search trees: distributional
  asymptotics for $m \ge 27$.
\newblock {\em Invited paper, Proceedings of the 7th Iranian Statistical
  Conference. Available via
  \url{http://www.ams.jhu.edu/~fill/papers/periodic.pdf}}, 7, 2004.

\bibitem{fill05a}
J.~A. Fill and N.~Kapur.
\newblock Transfer theorems and asymptotic distributional results for {$m$}-ary
  search trees.
\newblock {\em Random Structures Algorithms}, 26(4):359--391, 2005.

\bibitem{finkel74}
R.~A. Finkel and J.~L. Bentley.
\newblock Quad trees: {A} data structure for retrieval on composite keys.
\newblock {\em Acta Inf.}, 4:1--9, 1974.

\bibitem{flajolet95}
P.~Flajolet, G.~Labelle, L.~Laforest, and B.~Salvy.
\newblock Hypergeometrics and the cost structure of quadtrees.
\newblock {\em Random Structures Algorithms}, 7(2):117--144, 1995.

\bibitem{fuchs14a}
M.~Fuchs.
\newblock A note on the quicksort asymptotics.
\newblock {\em Random Structures Algorithms}, 46(4):677--687, 2015.

\bibitem{fuchs16}
M.~Fuchs and H.-K. Hwang.
\newblock Dependence between size and external path-length in random tries.
\newblock preprint, 2016.

\bibitem{grubel14}
R.~Gr{\"u}bel.
\newblock Search trees: metric aspects and strong limit theorems.
\newblock {\em Ann. Appl. Probab.}, 24(3):1269--1297, 2014.

\bibitem{kabluchko14}
R.~Gr\"ubel and Z.~Kabluchko.
\newblock A functional central limit theorem for branching random walks, almost
  sure weak convergence, and applications to random trees.
\newblock preprint, 2014.

\bibitem{holmgren11}
C.~Holmgren.
\newblock A weakly 1-stable distribution for the number of random records and
  cuttings in split trees.
\newblock {\em Adv. in Appl. Probab.}, 43(1):151--177, 2011.

\bibitem{hwang03}
H.-K. Hwang.
\newblock Second phase changes in random {$m$}-ary search trees and generalized
  quicksort: convergence rates.
\newblock {\em Ann. Probab.}, 31(2):609--629, 2003.

\bibitem{janson04}
S.~Janson.
\newblock Functional limit theorems for multitype branching processes and
  generalized {P}\'olya urns.
\newblock {\em Stochastic Process. Appl.}, 110(2):177--245, 2004.

\bibitem{janson06b}
S.~Janson.
\newblock Congruence properties of depths in some random trees.
\newblock {\em ALEA Lat. Am. J. Probab. Math. Stat.}, 1:347--366, 2006.

\bibitem{janson08}
S.~Janson and R.~Neininger.
\newblock The size of random fragmentation trees.
\newblock {\em Probab. Theory Related Fields}, 142(3-4):399--442, 2008.

\bibitem{knape14}
M.~Knape and R.~Neininger.
\newblock P\'olya urns via the contraction method.
\newblock {\em Combin. Probab. Comput.}, 23(6):1148--1186, 2014.

\bibitem{knuth98}
D.~E. Knuth.
\newblock {\em The {A}rt of {C}omputer {P}rogramming. {V}ol. 3. {S}orting and
  {S}earching.}
\newblock Addison-Wesley, Reading, MA, 1998.
\newblock Second edition.

\bibitem{lew94}
W.~Lew and H.~M. Mahmoud.
\newblock The joint distribution of elastic buckets in multiway search trees.
\newblock {\em SIAM J. Comput.}, 23(5):1050--1074, 1994.

\bibitem{mahmoud86}
H.~M. Mahmoud.
\newblock On the average internal path length of {$m$}-ary search trees.
\newblock {\em Acta Inform.}, 23(1):111--117, 1986.

\bibitem{mahmoud92}
H.~M. Mahmoud.
\newblock {\em {Evolution of Random Search Trees}}.
\newblock Wiley-Interscience. John Wiley \&amp; Sons, Inc., New York, 1992.
\newblock A Wiley-Interscience Publication.

\bibitem{mahmoud89}
H.~M. Mahmoud and B.~Pittel.
\newblock Analysis of the space of search trees under the random insertion
  algorithm.
\newblock {\em J. Algorithms}, 10(1):52--75, 1989.

\bibitem{mailler14}
C.~Mailler.
\newblock Describing the asymptotic behaviour of multicolour {P}\'olya urns via
  smoothing systems analysis.
\newblock {\em arXiv:1407.2879}, 2014.

\bibitem{majumdar05}
S.~N. Majumdar, D.~S. Dean, and P.~L. Krapivsky.
\newblock Understanding search trees via statistical physics.
\newblock {\em Pramana}, 64:1175--1189, 2005.

\bibitem{munsonius11}
G.~O. Munsonius.
\newblock On the asymptotic internal path length and the asymptotic {W}iener
  index of random split trees.
\newblock {\em Electron. J. Probab.}, 16:no. 35, 1020--1047, 2011.

\bibitem{muntz71}
R.~Muntz and R.~Uzgalis.
\newblock Dynamic storage allocation for binary search trees in a two-level
  memory.
\newblock {\em Proceedings of the Princeton Conference on Information Sciences
  and Systems}, 4:345--349, 1971.

\bibitem{neininger01}
R.~Neininger.
\newblock On a multivariate contraction method for random recursive structures
  with applications to {Q}uicksort.
\newblock {\em Random Structures Algorithms}, 19(3-4):498--524, 2001.

\bibitem{neininger14}
R.~Neininger.
\newblock Refined quicksort asymptotics.
\newblock {\em Random Structures Algorithms}, 46(2):346--361, 2015.

\bibitem{neininger99}
R.~Neininger and L.~R\"uschendorf.
\newblock On the internal path length of {$d$}-dimensional quad trees.
\newblock {\em Random Structures Algorithms}, 15(1):25--41, 1999.

\bibitem{neininger04}
R.~Neininger and L.~R\"uschendorf.
\newblock A general limit theorem for recursive algorithms and combinatorial
  structures.
\newblock {\em Ann. Appl. Probab.}, 14(1):378--418, 2004.

\bibitem{rosler01}
U.~R\"osler.
\newblock On the analysis of stochastic divide and conquer algorithms.
\newblock {\em Algorithmica}, 29(1-2):238--261, 2001.

\bibitem{sulzbach14}
H.~Sulzbach.
\newblock On martingale tail sums for the path length in random trees.
\newblock {\em http://arXiv:1412.3508}, 2014.

\end{thebibliography}

% \appendix
% \addcontentsline{toc}{section}{Appendices}
% \section*{Appendices}

%\pagebreak
%\pagestyle{empty}

\end{document}